\documentclass{article}
\usepackage[final]{microtype}

\usepackage[utf8]{inputenc}
\usepackage{amsmath,amsfonts,amssymb,amscd,amsthm,xspace}
\usepackage{mathtools}
\usepackage{subfiles}
\usepackage{tikz-cd}
\usepackage{tikz}
\usetikzlibrary{calc}
\usetikzlibrary{matrix}
\usepackage{extarrows}
\usepackage{float}
\usepackage{bm}
\usepackage{pgf}
\usepackage{color}
\usepackage{subcaption}
\usepackage{import}

\usepackage{hyperref}

\usepackage[backend=bibtex,style=numeric,doi=false,isbn=false,url=false,date=year,maxbibnames= 10]{biblatex}
\AtEveryBibitem{%
  \clearfield{note}%
}
\renewbibmacro{in:}{}
\DeclareFieldFormat[article,book,inproceedings,thesis,incollection,inbook]{title}{#1} 

\bibliography{ref.bib}

\usepackage{cleveref}
\usepackage{bbm}

\newtheorem{theorem}{Theorem}[section]

\newtheorem{lemma}[theorem]{Lemma}
\newtheorem{proposition}[theorem]{Proposition}

\theoremstyle{definition}
\newtheorem{definition}[theorem]{Definition}
\theoremstyle{remark}
\newtheorem{remark}[theorem]{Remark}

\newcommand{\R}{\mathbb{R}}
\newcommand\dif{\mathop{}\!\mathrm{d}}

\newcommand{\spin}{\mathrm{Spin}}
\newcommand{\so}{\mathrm{SO}}
\renewcommand{\H}{\mathbb{H}}
\newcommand{\qi}{\mathbbm{i}}
\newcommand{\qj}{\mathbbm{j}}
\newcommand{\qk}{\mathbbm{k}}
\newcommand{\s}{\mathrm{s}}

\newcommand{\spann}{\mathop{\mathrm{span}}}%

\title{A discrete extrinsic and intrinsic Dirac operator}
\author{Tim, Hoffmann \and  Zi, Ye\footnote{This research was supported by the DFG-Collaborative Research Center, TRR 109, "Discretization in Geometry and Dynamics"}}
\begin{document}
\maketitle
\begin{abstract}
  In differential geometry of surfaces the Dirac operator appears intrinsically as a tool to address the immersion problem as well as in an extrinsic flavour (that comes with spin transformations to comformally transfrom immersions) and the two are naturally related.
  In this paper we consider a corresponding pair of discrete Dirac operators, the latter on discrete surfaces with polygonal faces and normals defined on each face, and show that many key properties of the smooth theory are preserved. In particular, the corresponding spin transformations, conformal invariants for them, and the relation between this operator and its intrinsic counterpart are discussed.

\end{abstract}
\section{Introduction}
The Dirac operator for Riemannian manifolds was originally constructed by Atiyah and Singer as an example for their index theorem (we will give the defintion of the Dirac operator in \cref{sec:relation}, for more details please refer to \cite{Atiyah_1968,lawson_spin_1990}). Since then a wide range of applications have been discovered in geometry, topology and physics. In particular people found it a viable way to deal with the immersion problem of manifolds, e.g., the immersion problem of surfaces in $\R^3$, $S^3$, and $\R^4$ (see \cite{friedrich_spinor_1998,roth_spinorial_2010}): Suppose $X$ is a surface with a metric, then a solution to the Dirac equation
\[D \phi = H \phi \]
of unit length corresponds to an  isometric immersion in $\R^3$ with mean curvature $H$.

On the other hand a Dirac-type operator
\begin{equation}
\label{eqn:ex_dirac}
D_f = -\frac{\dif f \wedge \dif }{\lvert \dif f \rvert^2}
\end{equation}
was developed to study conformal transformations of immersed surfaces in $\R^3$ \cite{kamberov_bonnet_1998}. Since it depends on a reference surface $f: X\rightarrow \R^3$, we call it the extrinsic Dirac operator. Any solutions to the equation
\[D_f \phi =\rho \phi\]
where $\rho$ is a real scalar, gives a new immersed surface by means of the spin transformation
\[\dif\tilde{f}=\overline{\phi}\cdot  \dif f  \cdot  \phi \]
with a mean curvature $(\rho+H)\lvert \phi \rvert^2$.
Given this, it is not surprising to see the following relation between the extrinsic and intrinsic Dirac operators:
\[D = D_f + H \]
Recently some beautiful numerical applications of $D_f$ have been created by Keenan, Pinkall and Schr\"oder  \cite{crane_robust_2013, crane_spin_2011}. Yet a solid mathematical discrete theory remains unknown.

In this paper we propose a discrete differential geometric framework for both the extrinsic and intrinsic Dirac operators. We will begin in  \cref{sec:extrinsic} with the extrinsic Dirac operator which is defined on a set of discrete surfaces called face-edge-constraint surfaces. The integrated mean curvature, which arises naturally in this setting by means of Steiner's formula, can be manipulated by the Dirac equation. One can use this idea to construct discrete minimal surfaces and their associated families, which turns out to be a generalization of the two types of minimal surfaces appearing in \cite{lam_discrete_2016}. Note that our discretization has induced some applications in computer graphics \cite{ye_2018}.\\
In the \cref{sec:intrinsic} we consider a more abstract intrinsic net, i.e., a cell complex with a length assigned to each edge. A discrete spinor bundle, together with a spinor connection, can then be constructed over this net. Furthermore, several results coming from the smooth theory can be shown to still hold in our setting: an even Euler characteristic implies the existence of a spin structure and the first Betti number determines the number of spin structures. The discrete intrinsic Dirac operator follows naturally and one can build a realization of the intrinsic net with prescribed integrated mean curvature in $\mathbb{R}^3$, which is a face-edge-constraint net, by solving the Dirac equation.\\

In the end we will see that just as in the smooth case, there is a nice connection between the extrinsic and intrinsic Dirac operators.

\section{Preliminaries:\\Quaterinionic interpretation of 3D rotations}
\label{sec:preliminaries}
We start by gathering some basic notions about quaternions and how they encode rotations in $\mathbb{R}^3$.
Let $\H$ denote the algebra of quaternions: the four dimensional real vector space $\H=\spann\{1,\qi,\qj,\qk\}$ together with the product relations $\qi^2=\qj^2=\qk^2=-1,\qi\qj=\qk,\qj\qk=\qi$, and $\qk\qi=\qj$. Then $\mathrm{Im}(\H):=\spann\{\qi,\qj,\qk\}$ is a three-dimensional subspace canonically isomorphic to $\R^3$ via 
\[(x,y,z)\mapsto x\qi+y\qj+z\qk.\] 
Given a vector $w\in\R^3$ the rotation of $w$ around a non-vanishing vector $u\in \R^3$ can be described in the following way: First let the vectors $w$ and $u$ be embedded in the imaginary quaternions in the above way. Then the rotation can be computed by:
\[\mathrm{R}^\theta_u (w)=q^{-1}\cdot w\cdot q\]
where $\mathrm{R}^\theta_u$ denotes the rotation of $w$ around $u$ through the angle $\theta$ and \[q=\lvert q\rvert \big(\cos \frac{\theta}{2}-\sin \frac{\theta}{2}\frac{u}{\lvert u\rvert}\big)\]Note that the angle $\theta$ is measured by the counterclockwise angle as one sees in the opposite direction of $u$.
\begin{lemma}
Let $w_1$, $w_2$ and $u$ be non-vanishing vectors in $\mathrm{Im}(\H)$ such that $\lvert w_1\rvert=\lvert w_2\vert$ and let $\theta \in (-\pi,\pi)$ denote the dihedral angle between two the planes $P_1= \spann\{w_1,u\}$ and $P_2=\spann\{w_2,u\}$. 
\begin{enumerate}
\item If $w_1-w_2\perp u$, then there is an uniquely defined unit quaternion $q$ such that
\begin{equation}
\label{eqn:length}
\mathrm{Im}(q)=\begin{cases}\frac{u}{\lvert  u\rvert}\lvert \mathrm{Im}(q)\rvert & \theta\neq 0 \\ 0 & \theta=0\end{cases}
\end{equation}
and 
\begin{equation}
\label{eqn:rotation}
q^{-1}\cdot w_1 \cdot q=w_2
\end{equation}
\item If $w_1+w_2\perp u$, then there is an unique real number $H$ such that 
\[(H+u)^{-1}\cdot w_1 \cdot (H+u)=-w_2\]
and we have
\[H=\lvert u\rvert \tan \frac{\theta}{2}.\]
\end{enumerate}
\end{lemma}
\begin{proof}
  Let $w_1,w_2, u$ and $\theta$ be as above.
\begin{enumerate}
\item Since $w_1-w_2\perp u$, $w_2$ can be obtained by rotating $w_1$ around $u$ by the angle $\theta$. There are two quaternions $q=\pm (\cos \frac{\theta}{2} - \sin\frac{\theta}{2}\frac{u}{\lvert u\rvert})$ satisfying \cref{eqn:rotation}, but only one of them 
\[q=\begin{cases} -\cos\frac{\theta}{2}+\sin \frac{\theta}{2} \frac{u}{\lvert u\rvert} & \sin\frac{\theta}{2} \geq 0 \\ \cos\frac{\theta}{2}-\sin \frac{\theta}{2} \frac{u}{\lvert u\rvert} & \sin \frac{\theta}{2}<0 \end{cases}\]
satisfies \cref{eqn:length}.
\item Since $w_1+w_2\perp u$, $-w_2$ can be obtained by rotating $w_1$ around $u$ by the angle $\theta+\pi$.
\begin{align*}
H+u &=\lvert u\rvert \tan \frac{\theta}{2}+u \\
&=\lvert u \rvert \Big(\tan \frac{\theta}{2} +\frac{u}{\lvert u \rvert}\Big)\\
&=\frac{\lvert u\rvert}{\cos \frac{\theta}{2}} \Big( \sin \frac{\theta}{2}+\cos \frac{\theta}{2} \frac{u}{\lvert u\rvert}\Big)\\
&=\frac{\lvert u\rvert}{\cos \frac{\theta}{2}} \Big( -\cos(\frac{\pi}{2} +\frac{\theta}{2})+\sin (\frac{\pi}{2}+\frac{\theta}{2}) \frac{u}{\lvert u\rvert}\Big) \\
&=-\frac{\lvert u\rvert}{\cos \frac{\theta}{2}} \Big( \cos(\frac{\pi+\theta}{2})-\sin (\frac{\pi+\theta}{2}) \frac{u}{\lvert u\rvert}\Big) 
\end{align*}
It follows that $(H+u)^{-1}\cdot w_1\cdot (H+u)=-w_2$ and it is also the unique quaternion with the imaginary part being exactly $u$.
\end{enumerate}
\end{proof}
\section{The extrinsic Dirac operator}
\label{sec:extrinsic}
Given an immersed smooth surface $f: X \rightarrow \R^3\subset \H$ and a smooth quaternion-valued function $\phi:X \rightarrow \H$, a smooth scale-rotation of every tangent plane can be constructed by (see \cite{kamberov_bonnet_1998,kamberov_quaternions_2002})
\begin{equation}
  \label{eq:scale-rotation}
  \widetilde{(\dif f)} =\overline{\phi}\cdot \dif f \cdot \phi
\end{equation}
If there exists a further smooth surface $\tilde f$ such that $\dif (\tilde f)=\widetilde{(\dif f)}$, then it follows that
\[0=\dif \dif (\tilde f)= \dif \widetilde{(\dif f)}=\dif (\overline{\phi}\cdot f \cdot \phi)\]
which gives the equation:
\begin{align*}
    \mathrm{D}_f (\phi)=0
\end{align*}
where $\mathrm{D}_f=\frac{\dif f\wedge \dif}{\lvert \dif f \rvert^2}$ is called the Dirac operator with respect to the immersion $f$. Since $\mathrm{D}_f$ depends on the immersion $f$ (and in order to distinguish it from the intrinsic Dirac operator by Atiyah), we call it extrinsic Dirac operator in the following context.

We are now interested in a discretization of $\mathrm{D}_f$. Note that a point-wise inner product $\langle \cdot,\cdot\rangle$ on the $1$-forms induced by the metric can be defined by
\[ \langle \omega,\eta\rangle \dif vol :=\omega \wedge * \eta \; .\]
Then $\mathrm{D}_f$ can be formally reformulated as 
\begin{equation}
\label{eqn:dirac_hodge}
D_f(\phi)=-\frac{\dif f \wedge \dif \phi}{|\dif f |^2}=\langle \dif f,*\dif \phi\rangle\; .
\end{equation}
Hence, in the discrete setting it is more natural to think of $\phi$ as the function of the dual vertices.\\
A net is a cell complex $X=(V,E,F)$ such that
\begin{enumerate}
    \item The faces are all polygons, but not necessary planar.
    \item The intersection of two adjacent faces contains always only one edge.
\end{enumerate}
By oriented nets we mean in every face we choose a preferred direction for every edge such that the common edge in two adjacent faces has the reversed direction (\cref{fig:orientation}). An immersed net is a net with each vertex assigned with a position in $\R^3$. The notation $df_{ij}$ indicates the immersed edge incident to the faces $\Delta_i$ and $\Delta_j$ and with the orientation in face $\Delta_i$. It is clear that
\[df_{ij}=-df_{ji}\]

\begin{figure}[h]
\centering
\def\svgwidth{0.5\columnwidth}
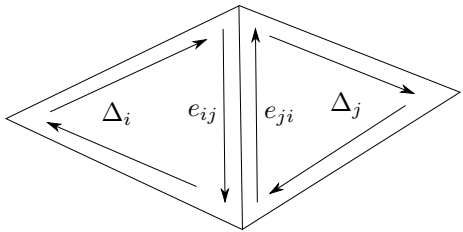
\caption{Orientation}\label{fig:orientation}
\end{figure}
{ Our basic object is the face-edge-constraint net, which looks similar as the one in \cite{hoffmann_discrete_2016}. Instead of considering normals at the vertices, we consider the normals defined on the faces. The vertex-based normals lead to a generalization of  several existing discrete integrable surfaces such as discrete integrable minimal surfaces and CMC surfaces. However, it is difficult to obtain the notion of the discrete mean curvature and corresponding Dirac operator in that setting. We will see in the following that the face-based normals would fill this gap. A generalization that merges these two types of edge-constraint nets is one of our goals for future research.}

\begin{definition}
A face-edge-constraint net $\mathfrak{X}=(X,f,n)$ is an oriented net $X=(V,E,F)$ with an immersion $f:V\rightarrow \R^3$  and unit normals $n:F\rightarrow \mathbb{S}^2$ assigned to each face, such that
\begin{align}
\label{eqn:edgeconstraint}
n_i+n_j \perp df_{ij}
\end{align}
holds for every pair of adjacent faces $\Delta_i$ and $\Delta_j$, where $df_{ij}:= f(v_j) -f(v_i)$ is the discrete $1$-form.
\end{definition}
\begin{remark}
\label{rem:classical}
An immersed oriented net with all faces being planar and $n_i$ being the normal of the face $\Delta_i$ is always a face-edge-constraint net. We call such nets classical nets.
\end{remark}
An advantage of the face-edge-constraint nets is that they come with a natural notion of mean curvature that arises from a face offset Steiner's formula, as we will see below. We are then able to introduce a discrete spin transformation and Dirac operator such that the Dirac equation guarantees the closing condition of the spin transformation. Moreover, one can control the mean curvature with the Dirac equation exactly as in the smooth case.
\begin{definition}
\label{def:dihedral}
Given a face-edge-constraint net the dihedral angle $\theta_{ij}$ from the face $\Delta_i$ to $\Delta_j$ is defined to be the angle from the plane $P_i$ to $P_j$, where $P_i=\spann\{n_i,df_{ij}\}$ and $P_j=\spann\{n_j,df_{ij}\}$ (\cref{fig:dihedral}). 
\begin{figure}[h]
\centering
\def\svgwidth{0.5\columnwidth}
\begingroup%
  \makeatletter%
  \providecommand\color[2][]{%
    \errmessage{(Inkscape) Color is used for the text in Inkscape, but the package 'color.sty' is not loaded}%
    \renewcommand\color[2][]{}%
  }%
  \providecommand\transparent[1]{%
    \errmessage{(Inkscape) Transparency is used (non-zero) for the text in Inkscape, but the package 'transparent.sty' is not loaded}%
    \renewcommand\transparent[1]{}%
  }%
  \providecommand\rotatebox[2]{#2}%
  \newcommand*\fsize{\dimexpr\f@size pt\relax}%
  \newcommand*\lineheight[1]{\fontsize{\fsize}{#1\fsize}\selectfont}%
  \ifx\svgwidth\undefined%
    \setlength{\unitlength}{196.35965016bp}%
    \ifx\svgscale\undefined%
      \relax%
    \else%
      \setlength{\unitlength}{\unitlength * \real{\svgscale}}%
    \fi%
  \else%
    \setlength{\unitlength}{\svgwidth}%
  \fi%
  \global\let\svgwidth\undefined%
  \global\let\svgscale\undefined%
  \makeatother%
  \begin{picture}(1,0.70008507)%
    \lineheight{1}%
    \setlength\tabcolsep{0pt}%
    \put(0,0){\includegraphics[width=\unitlength,page=1]{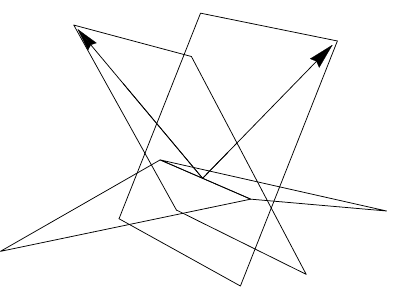}}%
    \put(0.20234816,0.16092933){\color[rgb]{0,0,0}\makebox(0,0)[lt]{\lineheight{1.25}\smash{\begin{tabular}[t]{l}$\Delta_i$\end{tabular}}}}%
    \put(0.74673287,0.26293462){\color[rgb]{0,0,0}\makebox(0,0)[lt]{\lineheight{1.25}\smash{\begin{tabular}[t]{l}$\Delta_j$\end{tabular}}}}%
    \put(0.34294118,0.48359712){\color[rgb]{0,0,0}\makebox(0,0)[lt]{\lineheight{1.25}\smash{\begin{tabular}[t]{l}$P_i$\end{tabular}}}}%
    \put(0.57249345,0.528046){\color[rgb]{0,0,0}\makebox(0,0)[lt]{\lineheight{1.25}\smash{\begin{tabular}[t]{l}$P_j$\end{tabular}}}}%
    \put(0.13867278,0.66139251){\color[rgb]{0,0,0}\makebox(0,0)[lt]{\lineheight{1.25}\smash{\begin{tabular}[t]{l}$n_i$\end{tabular}}}}%
    \put(0.80007151,0.63630491){\color[rgb]{0,0,0}\makebox(0,0)[lt]{\lineheight{1.25}\smash{\begin{tabular}[t]{l}$n_j$\end{tabular}}}}%
  \end{picture}%
\endgroup%

\caption{Dihedral angle}\label{fig:dihedral}
\end{figure}

\end{definition}
\begin{definition}
For a face-edge-constraint net the integrated mean curvature of the edge $e_{ij}$ is defined by 
\begin{equation}
\label{def:int_mc}
\mathbf{H}_{ij}=\frac{1}{2}\lvert df_{ij}\rvert \tan  \frac{\theta_{ij}}{2}\;.
\end{equation}
The mean curvature of a face is defined to be the sum of the mean curvatures of all the edges around the face:
\[\mathbf{H}_i=\sum\limits_j \mathbf{H}_{ij}\]
where $j$ runs through all the adjacent faces of $\Delta_i$.
\end{definition}
\begin{remark}
Suppose $X$ is a smooth immersed surface and $X_t$ is the surface offset obtained by shifting every point of $X$ along the normals with distance $t$. Then, Steiner's formula for the infinitesimal area $\dif A$ of $X_t$ gives
\begin{equation}
\dif A(X_t)=(1+2Ht+Kt^2) \dif A(X)
\end{equation}
where $H$ and $K$ stand for the mean curvature and Gauss curvature of $X$ respectively. In order to be consistent with the terminology in \cite{crane_spin_2011,kamberov_bonnet_1998}, we choose the sign of $H$ which is different from the one in \cite{karpenkov_offsets_2014}.\\
Now let us consider a classical face-edge-constraint net $\mathfrak{X}$. If we move the plane of the face $\Delta_i$ along $n_i$, as well as all the  faces $\Delta_j$ adjacent to $\Delta_i$ along $n_j$, with the distance $t$, then we obtain the face offset $\Delta_i^t$. The area of $\Delta_i^t$ is 
\begin{equation}
\mathrm{Area}(\Delta_i^t)=\Big(1+\frac{2\mathbf{H}_i}{\mathrm{Area}(\Delta_i)}t+\mathrm{o}(t^2)\Big)\mathrm{Area}(\Delta_i)
\end{equation}
hence our mean curvature can be thought of as the the mean curvature integrated over the face $\Delta_i$.
\end{remark}
\begin{proof}
See \cite[Thm 2.4]{karpenkov_offsets_2014}.
\end{proof}
The next definition ties together all the edge-located information in one quaternionic object:
\begin{definition}
The hyperedge $E_{ij}\in \H$ is a quaternion whose real part is the mean curvature of the edge $e_{ij}$ and whose imaginary part is the natural embedding of the edge into $\H$, i.e., 
\[E_{ij}:= \mathbf{2H}_{ij}+df_{ij}\]
\end{definition}
It is easy to see the following two properties of hyperedges:
\noindent
\begin{proposition}
For any hyperedge one finds:
\begin{enumerate}

    \item $E_{ij}=\overline{E_{ji}}$
    \item If the dihedral angle $\theta_{ij}=0$, then $E_{ij}=df_{ij}$ is purely imaginary.
\end{enumerate}
\end{proposition}
One can read hyperedges as rotation quaternions. This way we obtain
\begin{proposition}
\label{eqn:edgeconstraint2}
\[E_{ij}^{-1}\cdot n_i\cdot E_{ij}=-n_j\]
\end{proposition}
\begin{proof}
Direct computation yields
\begin{align*}
E_{ij} &= \tan \frac{\theta_{ij}}{2} \lvert df_{ij}\rvert +df_{ij} \\
&=\lvert df_{ij}\rvert \cos\tfrac{\theta_{ij}}{2}\left( \sin \tfrac{\theta_{ij}}{2}+ \cos \tfrac{\theta_{ij}}{2}\frac{df_{ij}}{\lvert df_{ij}\rvert}\right) \\
&=\lvert df_{ij}\rvert \cos\tfrac{\theta_{ij}}{2}\left( -\cos \tfrac{\theta_{ij}+\pi}{2}+ \sin \tfrac{\theta_{ij}+\pi}{2}\frac{df_{ij}}{\lvert df_{ij}\rvert}\right)\; .
\end{align*}
Apparently $n_i$ gets mapped to $-n_j$ by the rotation around the axis $\frac{df_{ij}}{\lvert df_{ij}\rvert}$ with the angle $\theta_{ij}+\pi$. 
\end{proof}
\begin{definition}  
Let $\mathcal{H}$ be the space of functions from the set of faces $F$ to $\H$. We also refer to the elements in $\mathcal{H}$ as the spinors. 
The discrete extrinsic Dirac operator, also denoted by $D_f$, is defined as follows:
\[D_f:\mathcal{H}\rightarrow \mathcal{H}\]
\[D_f(\phi)|_i=\sum\limits_j E_{ij}\cdot (\phi_j-\phi_i),\quad\text{for all faces $i\in F$},\]
where $j$ runs through all the neighboring faces of $i$.
\end{definition}

$D_f$ has a similar form as its smooth counterpart \cref{eqn:dirac_hodge}. Since the sum of the imaginary parts of hyperedges around a face vanishes, the Dirac operator can be rewritten as
\begin{align*}
    D_f(\phi)|_i &= \frac{1}{2}\sum\limits_{j} E_{ij}\cdot \phi_j- \frac{1}{2}\big(\sum\limits_{j}E_{ij}\big)\cdot \phi_i \\
    &=\frac{1}{2}\sum\limits_{j} E_{ij}\cdot \phi_j- \big(\sum\limits_{j}\mathbf{H}_{ij}\big)\phi_i \\
    &=\frac{1}{2}\sum\limits_{j} E_{ij}\cdot \phi_j- \mathbf{H}_i\phi_i
\end{align*}
\begin{proposition}
Let $\langle \cdot,\cdot\rangle$ be the scalar product defined on $\mathcal{H}$
\[\langle \phi,\psi\rangle =\sum_i\overline{\phi_i}\psi_i\]
where $i$ runs through all the faces of $X$ and suppose $X$ is a closed net. Then, the discrete extrinsic Dirac operator $D_f$ is self-adjoint.
\end{proposition}
\begin{proof}
Let $j_i$ be the indices of the faces neighbouring to $i$.
\begin{align*}
    \langle D_f \phi, \psi \rangle &= \sum\limits_i \overline{D_f \phi _i} \psi_i  \\
    &= \sum\limits_i \overline{\sum\limits_{j_i} \frac{1}{2}E_{ij}\cdot \phi_{j_i} -\mathbf{H}_i \phi _i}\psi_i\\
    &=\sum\limits_i \sum\limits_{j_i}\left( \overline{\phi_{j_i}}\overline{\frac{1}{2}E_{ij_i}}\psi_i-\mathbf{H}_i \overline{\phi_i}\psi_i \right)
\end{align*}
If $X$ is closed then we can switch the indices in the first term and it yields
\begin{align*}
    \langle D_f \phi, \psi \rangle
    &=\sum\limits_i \sum\limits_{j_i}\left(\frac{1}{2} \overline{\phi_{i}}\overline{E_{j_i i}}\psi_{j_i}-\mathbf{H}_i \overline{\phi_i}\psi_i \right)\\
    &=\sum\limits_i \sum\limits_{j_i}\left(\frac{1}{2} \overline{\phi_{i}}E_{ij_i}\psi_{j_i}-\mathbf{H}_i \overline{\phi_i}\psi_i \right)\\
    &= \sum\limits_i \overline{\phi_i}  \sum\limits_{j_i}\left( \frac{1}{2}E_{ij_i} \psi_{j_i} -\mathbf{H}_i \psi_i\right) \\
    &=\sum\limits_i \overline{\phi}_i D_f\psi_i \\
    &=\langle \phi,D_f\psi \rangle
\end{align*}
\end{proof}
We will now define a scale-rotation type of transformation for face-edge-constraint nets in the spirit of \eqref{eq:scale-rotation} together with a condition for the result to be integrable into a new face-edge-constraint net:
\begin{definition}
\label{def:spin_tr}
Let $\mathfrak{X}$ be a face-edge-constraint net. The discrete spin transformation $s_\phi$ with respect to $\phi$, which is a map from faces to quaternion $\phi:F\rightarrow \mathbb{H}$,  is given by (\cref{fig:spintransformation}):
\begin{align*}
\s_\phi(E_{ij}) &= \overline{\phi_i}\cdot E_{ij}\cdot \phi_{j}\\
\s_\phi(n_i) &= \phi_i^{-1} \cdot n_i \cdot \phi_i
\end{align*}
\end{definition}
\begin{figure}[h]
\centering
\def\svgwidth{0.9\columnwidth}
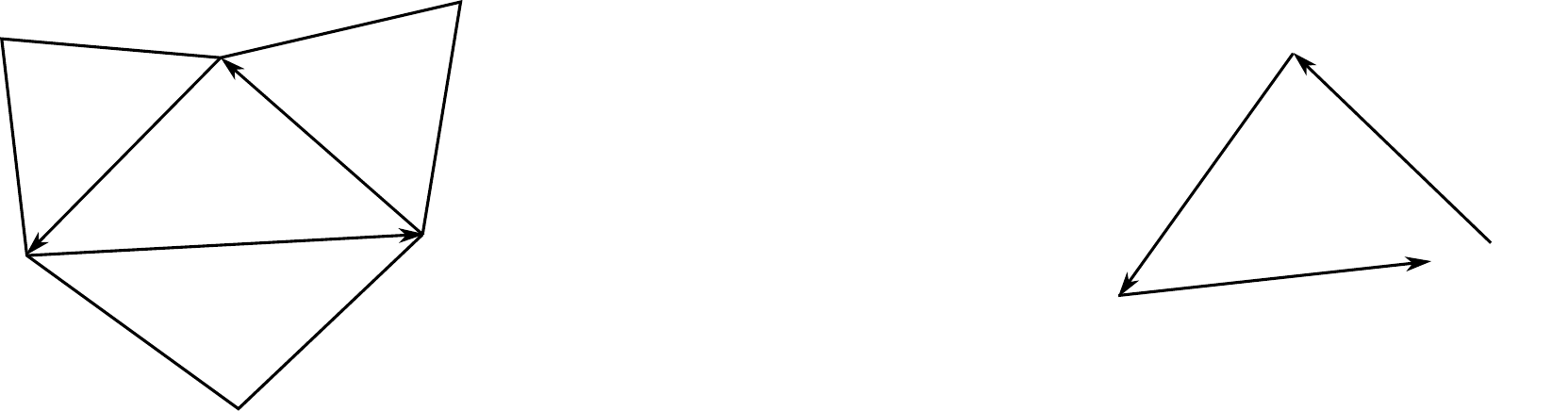
\caption{Discrete spin transformation}
\label{fig:spintransformation}
\end{figure}
\begin{theorem}
For a simply-connected net $\mathfrak{X}$, if 
\begin{equation}
  \label{eq:discrete-closedness}
  D_f \phi =\bm{\rho} \phi
\end{equation}
where $\bm{\rho}:F\rightarrow \R$ is a real function, then the imaginary parts of the hyperedges obtained by spin transformation are closed around every face.
\end{theorem}
\begin{proof}
The spin transformation of the face $\Delta_i$ is
\begin{align*}
\sum\limits_{j} s_\phi(E_{ij}) &=\sum\limits_j \overline{\phi_{i}}\cdot E_{ij}\cdot \phi_j \\
&=\overline{\phi_i}\cdot (\sum\limits_j E_{ij}\cdot \phi_j)\\
&=2\overline{\phi_j}(\bm{\rho}_i+\mathbf{H}_i)\phi_i\\
&=2(\bm{\rho}_i+\mathbf{H}_i) \lvert \phi_i\rvert^2\\
\end{align*}
which is a real number. Hence the imaginary parts of the transformed hyperedges add to zero. 
\end{proof}

The following proposition shows that the spin transformation maps a face-edge-constraint net again to a face-edge-constraint net. 
\begin{proposition}
  \label{pro:preserve}
  Let $s$ be a spin-transformation as above. Then
\[\s(E_{ij})^{-1}\cdot \s(n_i)\cdot \s(E_{ij})=-\s(n_j).\]
\end{proposition}
\begin{proof} A direct calculation yields:
\begin{align*}
    \s(E_{ij})^{-1}\cdot \s(n_i) \cdot \s(E_{ij}) &= (\overline{\phi_i}\cdot E_{ij}\cdot \phi_j)^{-1}\cdot \phi_i^{-1}\cdot n_i\cdot \phi_i\cdot (\overline{\phi_i}\cdot E_{ij}\cdot \phi_j)\\
    &=\phi_j^{-1}\cdot E_{ij}^{-1}\cdot n_j\cdot E_{ij}^{-1}\cdot \phi_j\\
    &=-\phi_j^{-1}\cdot n_j\cdot \phi_j \\
    &=-\s(n_j).
\end{align*}
\end{proof}
Let $\mathcal{X}$ be the space of all face-edge-constraint nets. For every $f\in \mathcal{X}$, every solution $\phi$ to \eqref{eq:discrete-closedness} gives rise to a new transformed face-edge-constraint net $\tilde f$. Its mean curvature $\widetilde{\mathbf{H}}$ changes from the original one $\mathbf{H}$ in the following way:
\begin{equation}
\label{eqn:meancurvature}
\widetilde{\mathbf{H}} = (\bm{\rho} + \mathbf{H}) \lvert \phi \rvert^2
\end{equation}
\begin{remark}
In smooth case we have the formula (see \cite{kamberov_bonnet_1998})
\begin{equation}
\label{eqn:changeofmch}
\tilde{H} \lvert \dif \tilde{f} \rvert = H \lvert \dif f \rvert +\rho \lvert \dif f \rvert  
\end{equation}
Let $h=H \lvert \dif  f \rvert $ be the mean curvature half-density, then \eqref{eqn:changeofmch} turns to
\begin{equation}
\label{eqn:smoothmch}
\tilde{h}=h+\rho \lvert  \dif f \rvert 
\end{equation}
Since the integrated mean curvature $\mathbf{H}$ is approximately $H \lvert \dif f\rvert^2$, we define the discrete mean curvature half-density by
\[\mathbf{h}_i:=\frac{\mathbf{H}_i}{\lvert \dif f\rvert} =\frac{\mathbf{H}_i}{\sqrt{\mathrm{Area}_i}}\]
then by $\widetilde{\mathrm{Area}_i} \approx \lvert \phi\rvert^4 \mathrm{Area}_i$ we have
\[\tilde{\mathbf{h}} \approx \mathbf{h}+\frac{\boldsymbol \rho}{\sqrt{\mathrm{Area}}}\]
therefore if we think of $\boldsymbol \rho$ as the integrated curvature potential, i.e., ${\boldsymbol \rho} \approx \rho \lvert \dif f \rvert^2$, it yields
\begin{equation}
\tilde{\mathbf{h}}_i \approx \mathbf{h} +\rho \lvert \dif f \rvert 
\end{equation}
which concides with the equation in smooth case \eqref{eqn:smoothmch}.
\end{remark}
\subsection{Minimal Surfaces and their Associated Family}
\begin{definition}
We call a face-edge-constraint net a minimal surface, if $\mathbf{H}_i = 0$ for all $i$.
\end{definition}
We know that if $\phi$ is a solution to the Dirac equation
\begin{equation}
\label{eqn:minimalsurface}
D_f \phi=-\mathbf{H} \phi
\end{equation}
then the spin transformation gives a minimal surface by \cref{eqn:meancurvature}. Recall that in smooth case a minimal surface doesn't come alone but always with an associated family \cite{Bobenko_1994}. In complete analogy, we will now see that there is a corresponding construction for face-edge-constraint minimal surfaces. Suppose $\phi_i$ is a solution to \eqref{eqn:minimalsurface}, then it is easy to verify that the following quaternionic functions parametrized by $\lambda$ all satisfy \eqref{eqn:minimalsurface} as well
\[\phi(\lambda)|_i = (\cos \lambda+\sin \lambda n_i) \cdot \phi_i\]
The explicit formula tor the associated family then is given by
\begin{align}
\label{eqn:family}    s(\lambda)(E_{ij}) &= \overline{\phi(\lambda)_i} E_{ij}\phi(\lambda)_j \\
    &=\overline{(\cos \lambda +\sin \lambda n_i)\phi_i}E_{ij}\cdot (\cos \lambda+\sin \lambda n_j)\phi_j \nonumber\\
    &=\overline{\phi_i}(\cos \lambda-\sin \lambda n_i)E_{ij}(\cos \lambda+\sin \lambda n_j)\phi_j \nonumber\\
    &=\overline{\phi_i}(\cos \lambda E_{ij}-\sin\lambda n_iE_{ij})(\cos \lambda+\sin \lambda n_j)\phi_j \nonumber\\
    &=\overline{\phi_i}(\cos^2\lambda E_{ij}+\cos \lambda\sin \lambda E_{ij}n_j \nonumber\\
    &\phantom{{}={}}-\sin\lambda\cos \lambda n_iE_{ij}-\sin^2\lambda n_iE_{ij}n_j)\phi_j\nonumber\\
    &=\overline{\phi_i}(\cos 2\lambda E_{ij}-\sin2\lambda n_iE_{ij})\phi_j \nonumber
\end{align}

In \cite{lam_discrete_2016} Lam shows that there exists an associated family which contains two types of well-known minimal surfaces, $A$-minimal surfaces coming from the discrete integrable system and $C$-minimal surfaces coming from an area variational approach.
 
\begin{definition}[\cite{lam_discrete_2016}]
\label{def:am}
Let $X = (V,E,F)$ be an oriented net with the immersion $f:V\rightarrow \mathbb{R}^3$ and unit vectors defined on faces $n:F\rightarrow \mathbb{S}^2$. $(X,f,n)$ is called an A-minimal surface if and only if
\begin{align}
\label{def:am1}df_{ij}\times (n_i-n_j) &= 0,\quad \text{for all edges $e_{ij}\in E$} \\
\label{def:am2}\langle n_i + n_j, df_{ij}\rangle &= 0,\quad \text{for all edges $e_{ij}\in E$.}
\end{align}
\end{definition}

\begin{definition}[\cite{lam_discrete_2016}]
\label{def:cm}
Let $X = (V,E,F)$ be an oriented net, $f:V\rightarrow \mathbb{R}^3$ be the immersion with planar faces and $n:F\rightarrow \mathbb{S}^2$ be the real face normals. Let $\theta_{ij}:=\angle (n_i,n_j)$ be the angle between neighbouring face normals. $(X,f,n)$ is called a C-minimal surface if and only if
\begin{equation}
\label{eqn:cm}
\mathbf{H}_i := \sum_j\lvert df_{ij}\rvert \tan \frac{\theta_{ij}}{2}
\end{equation}
vanishes for all faces $i\in F$, where $j$ runs through all neighboring faces of $i$. 
\end{definition}

\begin{theorem}
The two types of discrete minimal surfaces above are both special face-edge-constraint minimal nets.
\begin{enumerate}
\item The A-minimal surface is the face-edge-constraint minimal net with vanishing integrated mean curvature over edges, i.e.,  $\mathbf{H}_{ij} = 0$ for all edges $e_{ij}$.  
\item The C-minimal surface is the classical face-edge-constraint minimal surface.
\end{enumerate}
\end{theorem}
\begin{proof}\mbox{}
\begin{enumerate}
\item It is easy to see that the condition \eqref{def:am2} is our condition of face-edge-constraint \eqref{eqn:edgeconstraint}. Moreover,  \eqref{def:am1} and \eqref{def:am2} imply that the vectors $n_i$, $n_j$ and $df_{ij}$ are coplanar, hence the dihedral angles $\theta_{ij}$ (\cref{def:dihedral}) vanish for all edges $e_{ij}\in E$. Therefore $\mathbf{H}_{ij}=0$ for all edges $e_{ij}$.
\item Clearly, when $n_i$ is the real face normal, $(X,f,n)$ is always a face-edge-constraint net (we call it a classical net, \cref{rem:classical}) and \eqref{eqn:cm} only differs from our integrated mean curvature \eqref{def:int_mc} by a constant factor.
\end{enumerate}
\end{proof}

It is then not surprising to see that the associated family given in \cite{lam_discrete_2016} can be reformulated with our spin transformation \eqref{eqn:family}. Moreover, our face-edge-constraint minimal surface is a generalization of the minimal surfaces in \cite{lam_discrete_2016}.

\begin{remark}
\label{rem:min_gen}
While the definitions in \cite{lam_discrete_2016} can model the minimal surface with curvature-line parameterization (C-minimal surface), asymptotic parameterization (A-minimal surface) and their associated family, our new definition covers more general minimal surfaces with arbitrary parameterization. 
\end{remark}

\paragraph{A Weierstrass representation}
Recall that in \cite{bobenko_holomorphic_2016} Lam and Pinkall define a 
discrete holomorphic quadratic differential $q: E\rightarrow \mathrm{Im}(\mathbb{C})$ on a planar triangulated mesh $z:V\to\mathbb{C}$, $X=(V,E,F)$ in the complex plane and they show that this gives rise to the two types minimal surfaces (\cref{def:am} and \cref{def:cm}) by means of a discrete analogue of the Weierstrass representation. By \cref{rem:min_gen} we know that our face-edge-constraint minimal surfaces can be considered as a generalization of these two types minimal surfaces and indeed we can generalize the discrete holomorphic quadratic differential, removing the restriction on $q$ of being purely imaginary, and obtain the following generalized discrete holomorphic quadratic differential:
\begin{definition}
Given a planar net on the complex plane $z:V\rightarrow \mathbb{C}$. A holomorphic quadratic differential is a funtion $q: E\rightarrow \mathbb{C}$ such that
\[\sum_j q_{ij}=0, \quad \text{for all vertices } i\in V\]
\[\sum_j q_{ij}/\dif z(e_{ij})=0,\quad \text{for all vertices } i\in V,\]
where $j$ runs through all the neighboring vertices of $i$. 
\end{definition}
Now, we are going to show that this holomorphic quadratic differential always gives a family of minimal surfaces in a similar manner to \cite{bobenko_holomorphic_2016}.
\begin{theorem}
Let $z:V\rightarrow \mathbb{C}$ be a realization of a simply connected triangular mesh and $q:E\rightarrow \mathbb{C}$ a holomorphic quadratic differential. Then there exists a minimal face-edge-constraint net $\mathfrak{X}_q$:
\begin{align*}
E_{ij} &= \mathrm{Re} \left( q_{ij}+ \frac{q_{ij}}{i(z_j-z_i)}  \left(  (1-z_iz_j)\qi + i(1+z_iz_j) \qj + (z_i + z_j)\qk \right)\right) \\
n &= \frac{1}{\lvert z\rvert^2+1}\begin{pmatrix} 2\mathop\mathrm{Re} z \\ 2\mathop\mathrm{Im} z \\ \lvert z\rvert^2-1 \end{pmatrix}
\end{align*}
where $\mathop\mathrm{Re}$ means taking the real part of each component of the quaternion. 
\end{theorem}
\begin{proof}
To see that the imaginary parts of the hyperedges are closed around each face, we refer to the proof of Theorem 6.3 in \cite{bobenko_holomorphic_2016}. By direct computation we have
\[E_{ij}^{-1}\cdot n_i \cdot E_{ij} = -n_j\]
indicating that $\mathfrak{X}_q$ is  indeed a face-edge-constraint net. Note, that the integrated mean curvature for an edge is $H_{ij}=\mathop\mathrm{Re}(q_{ij})$, hence 
\[H_{i}= \sum_j \mathop\mathrm{Re}(q_{ij}) =0 \]
at any face $\Delta_i$ by assumption, showing that $\mathfrak{X}_q$ is minimal.
\end{proof}
\begin{remark}
We can construct the associated family of a minimal surface by rotating $q_{ij}$ with a constant unit complex number, $q_{ij}\rightarrow e^{\lambda i} q_{ij}$, which is basically equivalent to what we have done in \eqref{eqn:family}.
\end{remark}

\subsection{A Spin Multi-Ratio}
In this section we shall investigate an invariant of the spin transformation. It turns out that this invariant -- we will call it the spin multi-ratio -- actually fully characterizes face-edge-constraint nets up to spin equivalence.
\begin{definition}
A path in a net $X$ is a sequence of faces \[\gamma=(\gamma(1),\gamma(2),\ldots,\gamma(n))\] where $\gamma(i)$ and $\gamma(i+1)$ are neighbouring faces or $\overrightarrow{\gamma(i)\gamma(i+1)}\in E^*$. The length of the path is defined by the number of dual edges in the path, i.e.,
\[\lvert \gamma=(\gamma(1),\gamma(2),\cdots,\gamma(n))\rvert=n-1\]
Given a face-edge-constraint net $\mathfrak{X}=(X,f,n)$ the spin multi-ratio $\mathrm{cr}_\mathfrak{X}$ is a map from the set of all the paths to the quaternions 
\[\mathrm{cr}(\gamma)=\begin{cases} \overline{E_{\gamma(1),\gamma(2)}}^{-1}\cdot E_{\gamma(2),\gamma(3)}\cdot \ldots\cdot  E_{\gamma(n-1),\gamma(n)} & \lvert \gamma\rvert\text{ is even} \\ \overline{E_{\gamma(1),\gamma(2)}}^{-1}\cdot E_{\gamma(2),\gamma(3)}\cdot \ldots \cdot  \overline{E_{\gamma(n-1),\gamma(n)}}^{-1} & \lvert \gamma\rvert\text{ is odd} \end{cases} \]
\end{definition}
\begin{figure}[h]
\centering
\includegraphics[scale=0.2]{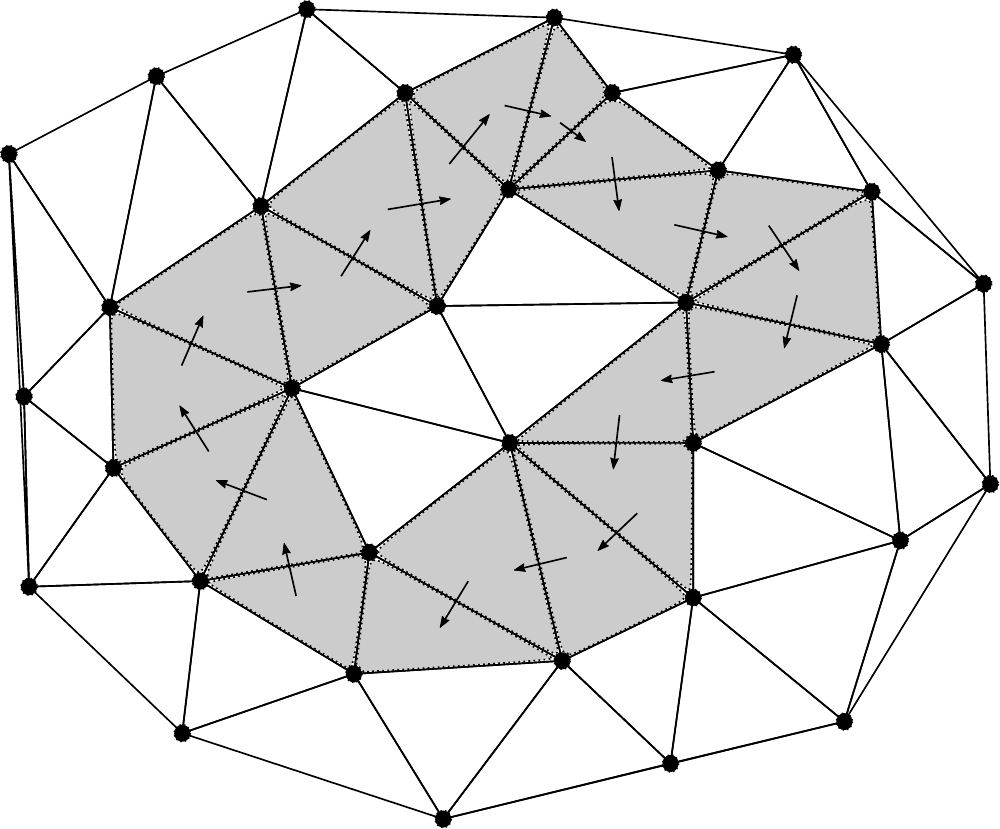}
\caption{A loop}\label{fig:loop}
\end{figure}
\begin{definition}
A loop at $\Delta_i$ is a path starting and ending both at the same face $\Delta_i$(\cref{fig:loop}). Let's define  an equivalence relation on the sets of all loops at $i$ by:
\[(\cdots,i,j,i,\cdots)\sim (\cdots,i,\cdots)\]
Then the set of all the loops at $\Delta_i$ modulo the equivalence relation is endowed with a group structure by:
\[\gamma_1\cdot \gamma_2=(\gamma_1(1),\gamma_1(2),\cdots,\gamma_1(n),\gamma_2(1),\gamma_2(2),\cdots,\gamma_2(m),\gamma_2(1))\]
where $\gamma_1=\big(\gamma_1(1),\cdots,\gamma_1(n),\gamma_1(1)\big)$ and $\gamma_2=\big(\gamma_2(1),\cdots,\gamma_2(m),\gamma_2(1)\big)$
and
\[\gamma_1^{-1}=\big(\gamma_1(1),\gamma_1(n),\cdots,\gamma_1(2),\gamma_1(1)\big)\]
We denote this group at $i$ by $\mathcal{O}_i$. Furthermore, $\mathcal{O}_i^{even}$ is the subgroup which consists of all the loops of even length at $i$, i.e.,
\[\mathcal{O}^{even}_i=\{\gamma\in \mathcal{O}_i|\; \lvert \gamma\rvert \text{ is even}\}\]
Note that the map $\mathrm{cr}_\mathfrak{X}$ restricted on $\mathcal{O}_i^{even}$ is a group homomorphism to $\H$.
\end{definition}
The next proposition shows how the spin multi-ratio changes under a spin transformation.
\begin{proposition}
Let $\mathrm{s}_\phi$ be the spin transformation
\[\mathrm{s}_\phi:\mathfrak{X}\mapsto \mathfrak{X}^\prime\]
with respect to the spinor $\phi$. Then
\[\mathrm{cr}_{\mathfrak{X}^\prime}(\gamma_i)=\begin{cases} \phi_i^{-1}\cdot \mathrm{cr}_\mathfrak{X}(\gamma_i)\cdot \phi_i & \lvert \gamma_i\rvert \text{ is even} \\ \phi_i^{-1}\cdot \mathrm{cr}_\mathfrak{X}(\gamma_i)\cdot \overline{\phi_i}^{-1} & \lvert \gamma_i\rvert  \text{ is odd}\end{cases}\]
Therefore the argument and the norm of the spin multi-ratio are preserved if the length of the loop is even.
\end{proposition}
From now on we simply index the faces in the loop by $\gamma=(1,2,\cdots,n,1)$.
\begin{remark}
\label{rem:length}
The norm of the spin multi-ratio contains the information of the edge length as well as the dihedral angles:
\begin{align*}
\lvert \mathrm{cr}(\gamma)\rvert &=\lvert E_{12}^{-1}\rvert\cdot \lvert E_{23}\rvert\cdot \cdots \cdot \lvert E_{n1}\rvert^{(-1)^n} \\
&=\lvert \cos\frac{\theta_{12}}{2}\rvert\cdot  \cdots \cdot \lvert \cos \frac{\theta_{n1}}{2}\rvert^{-1^{n+1}}\cdot \lvert e_{12}\rvert^{-1}\cdot  \cdots \cdot \lvert e_{n1}\rvert^{(-1)^n}
\end{align*}
\end{remark}
\begin{proposition}
\label{pro:normal}
For a loop $\gamma=(1,2,\cdots,n,1)$ of even length the axis of the spin multi-ratio $\mathrm{cr}(\gamma)$ is always parallel to the normal $n_1$ on $\gamma(1)$. For a loop with odd length the spin multi-ratio is always purely imaginary and perpendicular to $n_1$.
\end{proposition}
\begin{proof}
Consider the rotation of $n_1$ by $\mathrm{cr}(\gamma)$:
\[\mathrm{cr}(\gamma)^{-1}\cdot n_{1} \cdot \mathrm{cr}(\gamma)\]
it can be decomposed to successive rotations and each of these rotations takes the normal $n_{\gamma(i)}$ to the $-n_{\gamma(i+1)}$. Hence after an even number of rotations the normal $n_{1}$ comes back to itself, i.e., 
\[\mathrm{cr}(\gamma)^{-1}\cdot n_{i} \cdot \mathrm{cr}(\gamma)=n_{1}\]
Since $n_{1}$ is a fix point of rotation represented by $\mathrm{cr}(\gamma)$, the axis of $\mathrm{cr}(\gamma)$ is exactly $n_{1}$.\\
 In case of an odd number of rotations one ends up with
\[\mathrm{cr}(\gamma)^{-1}\cdot n_{1} \cdot \mathrm{cr}(\gamma)=-n_{1}\]
so $\mathrm{cr}(\gamma)$ must furnish a 180 degree rotation (thus it is purely imaginary) with an axis perpendicular to $n_1$.
\end{proof}
With \cref{rem:length} and \cref{pro:normal} we have a clear understanding of the geometric meaning of the norm and direction of the spin multi-ratio.  Next we are going to show some geometric interpretation of its argument. Since now we only care about the argument, we use a modified version of spin multi-ratio, denoted by $\hat{\mathrm{cr}}$, for the purpose of simplicity.%
\[\hat{\mathrm{cr}}(\gamma) :=E_{12}\cdot E_{23}\cdot \cdots \cdot E_{n1}\]
which differs from the true spin multi-ratio only by a scalar factor.\\
The rough idea is the following: one can rigidly unfold a classical net so that the spin multi-ratio would be factorized into two parts, both of  which are easily understood. If the net is not classical one can first project the edges onto the planes perpendicular to the normals and carry out the unfolding.
\begin{lemma}
\label{lem:h}
Let $df^i_{ij}$ be the pure imaginary quaternion with the same length as $E_{ij}$ and parallel to the projection of $df_{ij}$ onto the plane perpendicular to $n_i$, i.e., 
\[df^i_{ij}=\frac{\lvert E_{ij} \rvert \cdot \big(df_{ij}-\langle df_{ij},n_i\rangle n_i\big)}{\lvert df_{ij}-\langle df_{ij},n_i\rangle n_i \rvert }\]
Then $E_{ij}$ can be factorized into $E_{ij}=df^i_{ij}\cdot h_{ij}$, where $h_{ij}$ is the quaternion satisfying the following properties:
\begin{enumerate}
\item $h_{ij}$ is a unit quaternion with positive real part.
\item The axis of $h_{ij}$ is perpendicular both to $n_i$ and $n_j$.\hfill\refstepcounter{equation}\textup{(\theequation)}\label{eqn:h}
\item $h_{ij}^{-1}\cdot n_i\cdot h_{ij}=n_j$.
\end{enumerate}
\end{lemma}
\begin{proof}
It is easy to show that $\lvert df_{ij}^i\rvert=\lvert E_{ij}\rvert $ and hence $\lvert h_{ij}\rvert =1$. Then we have
\begin{align*}
    h_{ij} &= \epsilon \left(-df_{ij}+\langle  df_{ij},n_i\rangle n_i\right)\cdot E_{ij} \\
    &=\epsilon \left(-df_{ij}+\langle  df_{ij},n_i\rangle n_i\right) \cdot \left( \tan \frac{\theta_{ij}}{2} \lvert df_{ij} \rvert + df_{ij} \right)\\
    &=\epsilon \Bigg( \lvert df_{ij}\rvert^2-\langle df_{ij},n_i\rangle ^2 -\tan \frac{\theta_{ij}}{2}\lvert  df_{ij}\rvert df_{ij} +\tan \frac{\theta_{ij}}{2}\lvert df_{ij}\rvert \langle df_{ij},n_i\rangle n_i \\
    & + \langle df_{ij},n_i\rangle n_i \times df_{ij} \Bigg)
\end{align*}
where $\epsilon$ is some positive number. It follows that 
\[\mathop\mathrm{Re}(h_{ij})=\lvert df_{ij}\rvert^2-\langle df_{ij},n_i\rangle^2 =\lvert df_{ij}^i\rvert^2 >0 \]
and
\begin{align*}
&\langle \mathop\mathrm{Im}(h_{ij}) , n_i\rangle \\ &= \langle -\tan \frac{\theta_{ij}}{2}\lvert  df_{ij}\rvert df_{ij} +\tan \frac{\theta_{ij}}{2}\lvert df_{ij}\rvert\langle df_{ij},n_i\rangle n_i + \langle df_{ij},n_i\rangle n_i \times df_{ij},n_i\rangle \\
&=  -\tan \frac{\theta_{ij}}{2}\lvert  df_{ij}\rvert \langle df_{ij},n_i\rangle +\tan \frac{\theta_{ij}}{2}\lvert  df_{ij}\rvert \langle df_{ij},n_i\rangle \\
&=0.
\end{align*}
Note that $\mathop\mathrm{Im}(df_{ij}^i)\perp n_i$ and $ df_{ij}^i\cdot n_i\cdot(df_{ij}^i)^{-1}$ represents the transformation which rotates $n_i$ around the axis $\mathop\mathrm{Im}(df_{ij}^i)$ about $180$ degree, hence 
\[df_{ij}^i\cdot n_i\cdot(df_{ij}^i)^{-1}=-n_i\]
and therefore
\begin{align*}
h_{ij}^{-1}\cdot n_i \cdot h_{ij}  &= E_{ij}^{-1}\cdot df_{ij}^i \cdot n_i \cdot (df_{ij}^i)^{-1} \cdot E_{ij} \\
&= -E_{ij}\cdot n_i \cdot E_{ij} \\
&= n_j 
\end{align*}
which as well implies that
\[\mathop\mathrm{Im}(h_{ij}) \perp n_j\;.\]
\end{proof}

\begin{lemma}
Let $\gamma=(1,2,\cdots,n,1)$ be a loop. The modified spin multi-ratio can be written as
\[\hat{\mathrm{cr}}_\mathfrak{X}(\gamma)=\lvert \hat{\mathrm{cr}}_\mathfrak{X}(\gamma)\rvert (\mathfrak{e}_{12}\cdot \mathfrak{e}_{23}\cdot \cdots\cdot \mathfrak{e}_{n,1})\cdot (h_{12}\cdot \cdots \cdot h_{n,1})\]
where $\mathfrak{e}_{i,i+1}$ are pure imaginary quaternions such that $\mathfrak{e}_{i,i+1}\perp n_1$   and $n_1$ is the normal of the face $\gamma(1)$.  If $\mathfrak{X}$ is classical then 
\[\angle (\mathfrak{e}_{i-1,i},\mathfrak{e}_{i,i+1})=\angle (df_{i-1,i},df_{i,i+1}).\] 
\end{lemma}
\begin{proof}
Factorizing all the hyperedges $E_{ij}$ the spin multi-ratio becomes
\begin{align*}
\hat{\mathrm{cr}}_\mathfrak{X}(\gamma) =\;& E_{12}\cdot E_{23} \cdot \cdots \cdot E_{n,1}\\
=\;& df^1_{12}\cdot h_{12}\cdot df^2_{23} \cdot h_{23} \cdot \cdots \cdot df^n_{n,1}\cdot h_{n,1}\\
 =\;&df_{12}^1\cdot (h_{12}\cdot df^2_{23}\cdot h_{12}^{-1})\cdot (h_{12}h_{23}\cdot df_{34}^3 \cdot h_{23}^{-1}h_{12}^{-1})\cdot \cdots \\
                                       &\cdot (h_{12}h_{23}\cdot\cdots\cdot h_{n-1,n}\cdot df_{1,n}^n \cdot h_{n-1,n}^{-1}\cdot \cdots \cdot h_{23}^{-1}h_{12}^{-1})\\
  & \cdot(h_{12}\cdot h_{23}\cdot\cdots \cdot h_{n-1,n}h_{n,1}).
\end{align*}
Let 
\begin{equation}
\label{eqn:int_angle}
\mathfrak{e}_{i,i+1}=h_{12}\cdot \cdots\cdot h_{i-1,i}\cdot df_{i,i+1}\cdot h_{i-1,i}^{-1}\cdot \cdots \cdot h_{12}^{-1},
\end{equation}
then, by \eqref{eqn:h} we have $\mathfrak{e}_{i,i+1}\perp n_1$ and $\hat{\mathrm{cr}}(\gamma)$ has the form
\[\hat{\mathrm{cr}}_\mathfrak{X}(\gamma)=\lvert \hat{\mathrm{cr}}_\mathfrak{X}(\gamma)\rvert (\mathfrak{e}_{12}\cdot \mathfrak{e}_{23}\cdot \cdots\cdot \mathfrak{e}_{n,1})\cdot (h_{12}\cdot \cdots \cdot h_{n,1}).\]
If $\mathfrak{X}$ is classical, then $df^i_{ij}=df_{ij}$ and 
\[\angle(df_{i-1,i},df_{i,i+1})=\angle (df_{i-1,i},h_{i-1,i}\cdot df_{i,i+1}\cdot h_{i-1,i}^{-1})\]
because the axis of $h_{i-1,i}$ is parallel  to $df_{i,i-1}$. Applying the same rotation on $e_{i-1,i}$ and $h_{i-1,i}\cdot e_{i,i+1} \cdot h_{i-1,i}^{-1}$ we get
\[\angle (\mathfrak{e}_{i-1,i},\mathfrak{e}_{i,i+1})=\angle (df_{i-1,i},df_{i,i+1}).\] 
\end{proof}

Therefore, up to scaling, the spin multi-ratio can be written as the product of two factors: We call $ \mathfrak{e}_{12}\cdot \mathfrak{e}_{23}\cdot \ldots\cdot  \mathfrak{e}_{n,1}$ the {edge part and $ h_{12}\cdot h_{23}\cdot \ldots \cdot h_{n,1}$ the curvature part.
To understand the edge part we need the following lemma:
\begin{figure}[h]
\centering
\def\svgwidth{0.31\columnwidth}
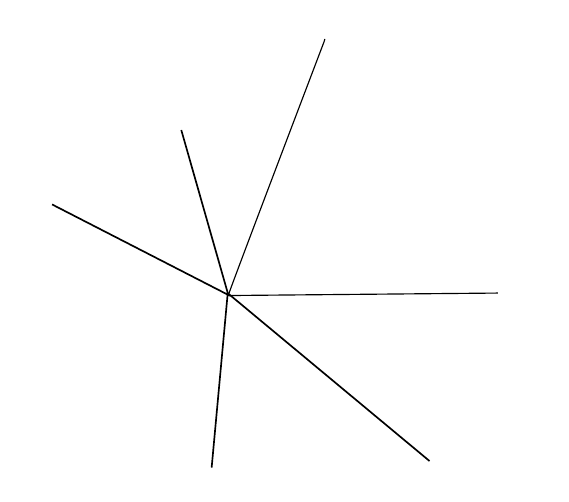
\caption{The product of quaternions in $\qi\qj$-plane}\label{fig:quat_prod}
\end{figure}

\begin{lemma}
\label{lem:prod}
Suppose $n$ is an even number. Let $q_1=\cos \omega_1\qi+\sin\omega_1 \qj$ and \[q_i=\cos (\omega_1-\sum\limits_{i=2}^{n}\omega_i)\qi+\sin(\omega_1-\sum\limits_{i=2}^{n}\omega_i)\qj\] (see \cref{fig:quat_prod}). Then
\[q_1\cdot q_2 \cdot \ldots q_n =\begin{cases}\cos(\Phi)+\sin(\Phi)\qk & n=0 \mod 4 \\ -\cos(\Phi)-\sin(\Phi)\qk & n=2 \mod 4 \end{cases}\]
where $\Phi= \sum\limits_{i=1}^{n/2}\omega_{2i}$.
\end{lemma}
We can prove the case that $n=2, 4$ by direct computation and generalize it by the induction.

Since $\mathfrak{e}_{i,i+1}$ are all coplanar, by \cref{lem:prod} we have:
\[\mathfrak{e}_{12}\cdot \mathfrak{e}_{23}\cdot \ldots\cdot  \mathfrak{e}_{n,1}=\pm( \cos(\Phi)+\sin(\Phi)\qk)\]
where $\Phi=\sum\limits^{n/2}_{i=1} \omega_{2i}$ and $\omega_i$ is the angle between the edges $\mathfrak{e}_{i-1,i}$ and $\mathfrak{e}_{i,i+1}$.
\subsubsection{The Argument of the Spin Multi-Ratio and the Angular Defect}
\label{subsec:angledefect}
\begin{figure}
\centering
\includegraphics[scale=0.2]{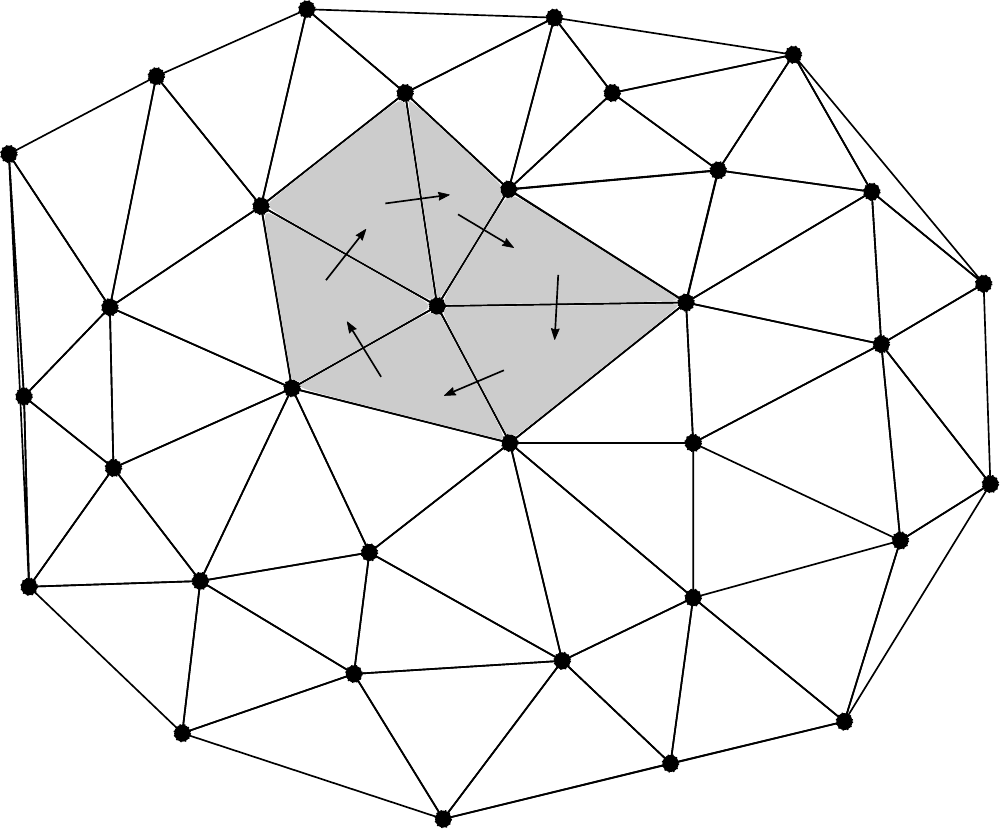}
\caption{A fundamental loop}\label{fig:floop}
\end{figure}
The angular defect around a vertex is known to be a polyhedral analog of Gaussian curvature and as such plays an important role in discrete differential geometry and we will show that it is closely related to the argument of the spin multi-ratio.

From now on we consider, for simplicity, a special set of loops which enclose only one vertex without duplicated dual edges. We call these loops fundamental(\cref{fig:floop}). The even fundamental loops are the fundamental loops enclosing a vertex with even degree. In the following $\mathrm{cr}(v)$ denotes the spin multi-ratio of the fundamental loop enclosing the vertex $v$. If no starting point of the fundamental loop is specified then $\mathrm{cr}(v)$ is well-defined up to conjugation in $\H$.\\ 
A vertex is called regular if and only if 
\begin{equation}
\label{def:regular}
\langle df_{i,i+1}\times df_{i-1,i} ,n_i\rangle>0
\end{equation}
holds for all incident edges. The angular defect of a regular vertex is defined by 
\[K(v)=2\pi-\sum\limits_{i=1}^{n} \omega_{i}\]
where $\omega_{i}$ is the angle between $\mathfrak{e}_{i-1,i}$ and $\mathfrak{e}_{i,i+1}$ defined in \eqref{eqn:int_angle}.
\begin{lemma}
Let $h_{i,i+1}$ be the quaternions satisfying the conditions \eqref{eqn:h}. Then 
\[h_{12}\cdot h_{23} \cdot \cdots \cdot h_{n,1}  = \cos\frac{K(v)}{2} + \sin\frac{K(v)}{2} n_1,\]
where $n_1$ is the normal of the first face $\gamma(1)$.
\end{lemma}
\begin{proof}
There are two unit quaternions, which differ by a sign, satisfying $h_{i,i+1}^{-1}\cdot n_{i} \cdot h_{i,i+1}=n_{i+1}$, so $h_{i,i+1}$ with positive real part is uniquely defined. Note that
\[h_{n,1}^{-1}\cdot \cdots \cdot h_{12}^{-1}\cdot n_1\cdot h_{12}\cdot \cdots \cdot h_{n,1}  =n_1\]
the axis of $h_{12}\cdot \cdots\cdot h_{n,1}$ is parallel to $n_1$ hence indeed \[h_{12}\cdot \cdots\cdot h_{n,1} \in \big\{a+b\cdot n_1\,|\,a,b\in \mathbb{R},a^2+b^2=1\big\}.\] 
If we cut along the edge $e_{n,1}$, fix the face $\Delta_1$ and unfold the faces along the path, then it gives a planar pattern, where the original edge $e_{n,1}$ incident to face $\Delta_1$ is denoted by $e_{n,1}^{1}$ and the edge $e_{n,1}$ incident to $\Delta_n$ is denoted by $e_{n,1}^{n}$. It follows that
\[h_{n,1}^{-1}\cdot \cdots \cdot h_{12}^{-1}\cdot e_{n,1}^1\cdot h_{12}\cdot \cdots \cdot h_{n,1}  =e_{n,1}^n\]
and hence
\[h_{12}\cdot \cdots \cdot h_{n,1}=\pm\left(\cos\frac{K(v)}{2} + \sin\frac{K(v)}{2} n_1\right).\]
To see that it indeed gives the right sign, observe that any pattern of vertex star can be deformed continuously to a planar pattern. Moreover we can always continously increase the angular defect while it's negative and decrease it while it's positive until $K=0$. During the deformation the value $h_{12}\cdot \cdots \cdot h_{n,1}$ changes  continuously until it becomes $1$ and it will never go through the value $-1$. Therefore we only have to check the sign for the planar pattern and the sign  of the other cases will be determined accordingly. In fact, the planar vertex star has $K =0$, and all $ h_{ij}$ would be just $1$. Hence we have 
\[h_{12}\cdot \cdots \cdot h_{n,1}=1.\]
\end{proof}
\begin{remark}
We can take the following example to visualize the map  
\begin{align*}
(-2\pi,2\pi)&\rightarrow \big\{a+b\cdot n_1\,|\,a,b\in \mathbb{R},a^2+b^2=1\big\},\\
\theta &\mapsto \cos \frac{\theta}{2} + \sin \frac{\theta}{2} n_1.
\end{align*}
 
\begin{figure}[h]
\centering
\begin{minipage}{0.4\textwidth}
\centering
\def\svgwidth{1\columnwidth}
\begingroup%
  \makeatletter%
  \providecommand\color[2][]{%
    \errmessage{(Inkscape) Color is used for the text in Inkscape, but the package 'color.sty' is not loaded}%
    \renewcommand\color[2][]{}%
  }%
  \providecommand\transparent[1]{%
    \errmessage{(Inkscape) Transparency is used (non-zero) for the text in Inkscape, but the package 'transparent.sty' is not loaded}%
    \renewcommand\transparent[1]{}%
  }%
  \providecommand\rotatebox[2]{#2}%
  \newcommand*\fsize{\dimexpr\f@size pt\relax}%
  \newcommand*\lineheight[1]{\fontsize{\fsize}{#1\fsize}\selectfont}%
  \ifx\svgwidth\undefined%
    \setlength{\unitlength}{416.19293745bp}%
    \ifx\svgscale\undefined%
      \relax%
    \else%
      \setlength{\unitlength}{\unitlength * \real{\svgscale}}%
    \fi%
  \else%
    \setlength{\unitlength}{\svgwidth}%
  \fi%
  \global\let\svgwidth\undefined%
  \global\let\svgscale\undefined%
  \makeatother%
  \begin{picture}(1,1.00394876)%
    \lineheight{1}%
    \setlength\tabcolsep{0pt}%
    \put(0,0){\includegraphics[width=\unitlength,page=1]{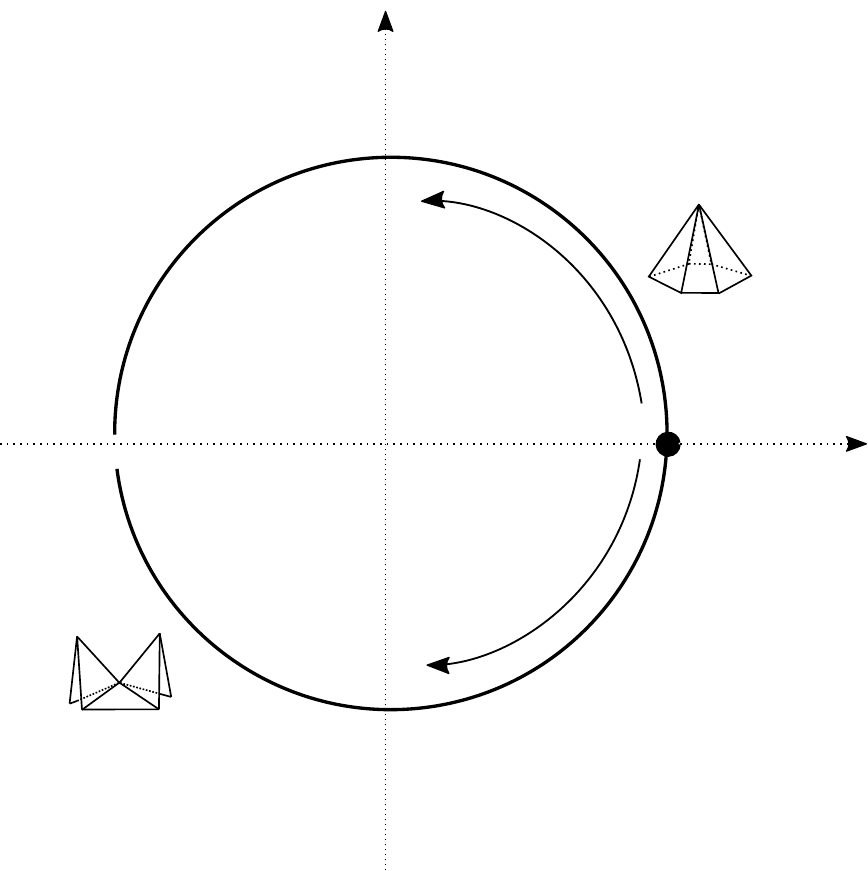}}%
    \put(0.78549205,0.60531102){\color[rgb]{0,0,0}\makebox(0,0)[lt]{\lineheight{1.25}\smash{\begin{tabular}[t]{l}\small $S_1$\end{tabular}}}}%
    \put(0.11031965,0.20680748){\color[rgb]{0,0,0}\makebox(0,0)[lt]{\lineheight{1.25}\smash{\begin{tabular}[t]{l}\small\\ $S_2$\end{tabular}}}}%
    \put(0.46183776,0.98569363){\color[rgb]{0,0,0}\makebox(0,0)[lt]{\lineheight{1.25}\smash{\begin{tabular}[t]{l}$n_1$\end{tabular}}}}%
    \put(0.79238788,0.50659279){\color[rgb]{0,0,0}\makebox(0,0)[lt]{\lineheight{1.25}\smash{\begin{tabular}[t]{l}$1$\end{tabular}}}}%
    \put(0.60811776,0.65218291){\color[rgb]{0,0,0}\makebox(0,0)[lt]{\lineheight{1.25}\smash{\begin{tabular}[t]{l}$+$\end{tabular}}}}%
    \put(0.6133491,0.34950969){\color[rgb]{0,0,0}\makebox(0,0)[lt]{\lineheight{1.25}\smash{\begin{tabular}[t]{l}$-$\end{tabular}}}}%
    \put(0.16473622,0.51347923){\color[rgb]{0,0,0}\makebox(0,0)[lt]{\lineheight{1.25}\smash{\begin{tabular}[t]{l}$-1$\end{tabular}}}}%
  \end{picture}%
\endgroup%

\end{minipage}%
\begin{minipage}{0.4\textwidth}
\centering
\def\svgwidth{1\columnwidth}
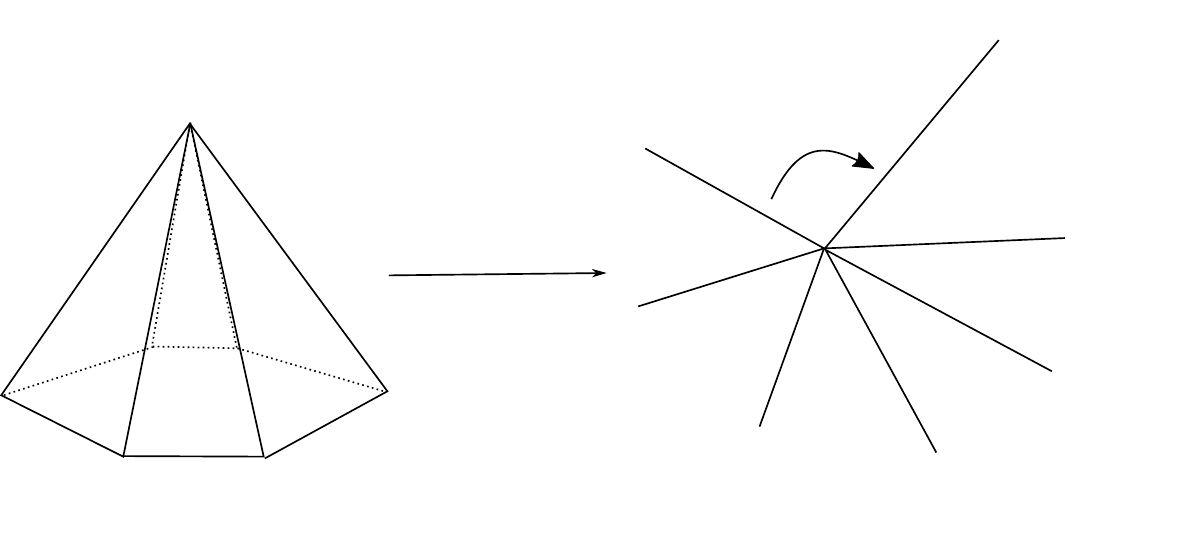\vspace{1cm}
\def\svgwidth{1\columnwidth}
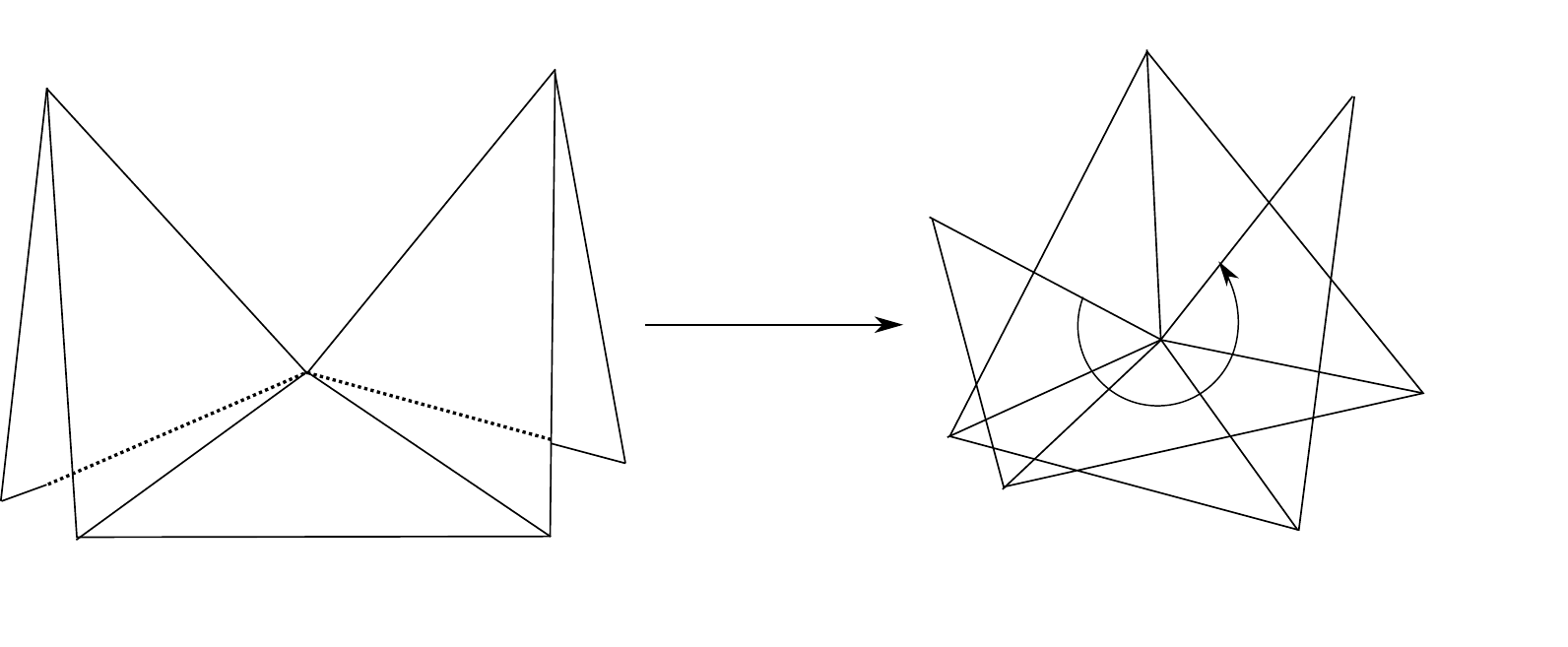
\end{minipage}
\caption{A sketch of the map $K(v)\mapsto \cos \frac{K(v)}{2}+ \sin\frac{K(v)}{2}n_1$.}\label{fig:angledefect}
\end{figure}
Assuming that two vertex stars $S_1$ and $S_2$ in \cref{fig:angledefect} have the same rotation angle between $e_{n,1}^1$ and $e_{n,1}^n$, we can  determine their positions up to the antipodal points on the circle. Observe that $S_1$  can be deformed to the planar vertex star without going through any pattern with angular defect $\pm \pi$, which are corresponding to the points $\pm n_1$ on the circle. Hence $S_1$ should sit in the first quadrant. By the analogous argument $S_2$ should sit in the third quadrant. 
\end{remark}
As a result the the argument of the spin multi-ratio can be characterized as follows:
\begin{theorem}
Suppose $\gamma = (1,2,\cdots,n,1)$ is a loop of even length. The spin multi-ratio can be written as
\[\frac{\mathrm{cr}_{\mathfrak{X}}(\gamma)}{\lvert \mathrm{cr}_{\mathfrak{X}}(\gamma)\rvert} = \pm\Big(\cos \frac{\Phi}{2}+\sin \frac{\Phi}{2} n_1\Big)\]
where $\Phi=K(v)+2\cdot \sum\limits_{i=1}^{n/2}\omega_{2i}$ and $n_1$ is the normal of the face $\gamma(1)$.
\end{theorem}
\begin{remark}
The argument of a vertex star with angular defect $K$ is the sum of the angles for the shaded regions in \cref{fig:argument}.\\
Since $K = 2\pi -\sum\limits_{i=1}^{n} \omega_i$, we can rewrite the argument as the alternating sum of the angles $\omega_i$:
\[\Phi = 2\pi +\sum_{i=1}^n(-1)^i \omega_i\; .\]
\end{remark}
\begin{figure}[h]
\centering
\def\svgwidth{0.5\columnwidth}
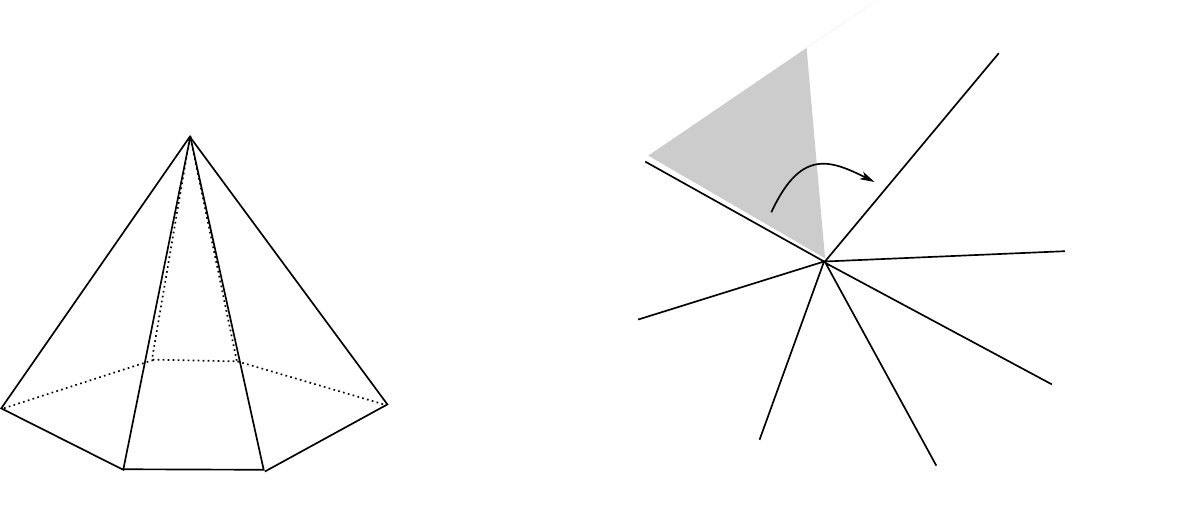
\caption{The argument of the spin multi-ratio.}\label{fig:argument}
\end{figure}
\subsection{Spin equivalence}
We are now able to show, that the spin multi-ratio determines the net up to spin transformations.
\begin{definition}
Given two face-edge-constraint nets $\mathfrak{X}$ and $\mathfrak{X}^\prime$ if there exists a spinor $\phi$ with $\lvert \phi_i\rvert \neq 0$ for all $i$ such that 
\[\mathrm{s}_\phi(\mathfrak{X})=\mathfrak{X}^\prime,\]
where $s_\phi$ is the spin transformation introduced in \cref{def:spin_tr},
then we say that $\mathfrak{X}$ and $\mathfrak{X}^\prime$ are spin equivalent.
\end{definition}
\begin{theorem}
Given two face-edge-constraint nets $\mathfrak{X}$ and $\mathfrak{X}^\prime$, if 
$\mathrm{cr}_\mathfrak{X}(\gamma)$ and $\mathrm{cr}_{\mathfrak{X}^\prime}(\gamma)$
have the same argument and norm for all $\gamma\in \mathcal{O}^{even}_i$ then they are spin equivalent. Moreover, if all the vertices in $X$ have even degree then there are a family of the spinor $\phi_\lambda$, parametrized by $S^1$, giving the spin transformation between $\mathfrak{X}$ and $\mathfrak{X}^\prime$. If there exists at least one vertex with odd degree then the spinor is unique. 
\end{theorem}
\begin{proof}
First consider the case with only even degree vertices, then all the loops have even length. Choose the $\phi_i$ such that $n^\prime_i=\phi_i^{-1}\cdot n_i\cdot \phi_i$. Note that all the possible choices form a $S^1$-parametrized set. Now we want to determine the value at the face $j$. First take a path from $i$ to $j$
\[\gamma=(i=1,2,\cdots,n=j)\]
and by induction let 
\begin{equation}
\label{eqn:induction}
\phi_{m+1}=E^{-1}_{m,m+1}\cdot \overline{\phi_{m}^{-1}}\cdot E_{m,m+1}^\prime
\end{equation}
for $m=1,\cdots, n-1$. Now, we just need to check that the value of $\phi$ is independent on the choice of path. Suppose $\gamma_1$ and $\gamma_2$ , with $\lvert \gamma_1\rvert=m_1$ and $\lvert \gamma_2\rvert =m_2$ are two paths connecting $i$ and $j$. Label the in-between vertices by:
\[\gamma_1=(i=\gamma_1(1),\gamma_1(2),\cdots,\gamma_1(m_1)=j)\]
and
\[\gamma_2=(i=\gamma_2(1),\gamma_2(2),\cdots, ,\cdots,\gamma_2(m_2)=j).\]
Since $\lvert \gamma_1\rvert+\lvert \gamma_2\rvert$ is even, $\lvert \gamma_1\rvert$ and $\lvert \gamma_2\rvert $ are either both even or both odd. Suppose that they are both even, then computing the value of $\phi_j$ along the path $\gamma_1$ we obtain that
\[\phi_j^\prime=E^{-1}_{\gamma_1(m_1-1),\gamma_1(m_1)}\cdot \cdots \cdot \overline{E_{\gamma_1(1),\gamma_2(2)}}\cdot \phi_i\cdot \overline{E^\prime_{\gamma_1(1),\gamma_1(2)}}^{-1}\cdots\cdot E^\prime_{\gamma_1(m_1-1),\gamma_1(m_1)}. \]
Then, computing the value of $\phi_j$ along the path $\gamma_2$ we find
\[\phi_j^{\prime\prime}=E^{-1}_{\gamma_2(m_2-1),\gamma_2(m_2)}\cdot \cdots \cdot \overline{E_{\gamma_2(1),\gamma_2(2)}}\cdot \phi_i\cdot \overline{E^\prime_{\gamma_2(1),\gamma_2(2)}}^{-1}\cdots\cdot E^\prime_{\gamma_2(m_2-1),\gamma_2(m_2)}. \]
Note that $\gamma_1\cdot \gamma_2^{-1}$ forms an even loop, so $\mathrm{cr}_\mathfrak{X}(\gamma_1\cdot\gamma_2^{-1})$ and $\mathrm{cr}_{\mathfrak{X}^\prime}(\gamma_1\cdot \gamma_2^{-1})$ have the same argument and norm. Besides, the axis of $\mathrm{\mathfrak{X}}(\gamma_1\cdot \gamma_2^{-1})$ is parallel to $n_i$ and the axis of $\mathrm{cr}_{\mathfrak{X}^\prime}(\gamma_1\cdot \gamma_2^{-1})$ is parallel to $n^\prime_i$. Therefore we have
\[\phi_i^{-1}\cdot \mathrm{cr}_{\mathfrak{X}}(\gamma_1\cdot \gamma_2^{-1})\cdot \phi_i =\mathrm{cr}_{\mathfrak{X}^\prime}(\gamma_1\cdot \gamma_2^{-1})\]
where by definition
\[\mathrm{cr}_{\mathfrak{X}}(\gamma_1\cdot \gamma_2^{-1})=\overline{E_{\gamma_1(1),\gamma_1(2)}}^{-1}\cdot E_{\gamma_1(2),\gamma_1(3)}\cdot \cdots \cdot E_{\gamma_2(2),\gamma_2(1)}\]
and
\[\mathrm{cr}_{\mathfrak{X}^\prime}(\gamma_1\cdot \gamma_2^{-1})=\overline{E^\prime_{\gamma_1(1),\gamma_1(2)}}^{-1}\cdot E^\prime_{\gamma_1(2),\gamma_1(3)}\cdot \cdots \cdot E^\prime_{\gamma_2(2),\gamma_2(1)}.\]
It then follows that
\begin{align*}
    \phi_i^{\prime\prime} &= E^{-1}_{\gamma_2(m_2-1),\gamma_2(m_2)}\cdot \cdots \cdot \overline{E}_{\gamma_2(1),\gamma_2(2)}\cdot \mathrm{cr}^{-1}_\mathfrak{X}(\gamma_1\cdot \gamma_2^{-1})\cdot  \phi_i \\
    &\cdot \mathrm{cr}_{\mathfrak{X}^\prime}(\gamma_1\cdot \gamma_2^{-1})\cdot  \overline{E^\prime}_{\gamma_2(1),\gamma_2(2)}^{-1}\cdots\cdot E^\prime_{\gamma_2(m_2-1),\gamma_2(m_2)}\\
    &= \phi_i^\prime.
\end{align*}
The argument is analogous for the case of $\lvert \gamma_1\rvert$ and $\lvert \gamma_2\rvert$ both being odd. \\
If there exists an odd loop $\gamma_o\in \mathcal{O}_i$, then we can first determine all the values of $\phi$ lying on the loop $\gamma_o$ by \eqref{eqn:induction}. Since $\mathrm{cr}_\mathfrak{X}(\gamma_o)$ and $\mathrm{cr}_{\mathfrak{X}^\prime}(\gamma_o)$ are both pure imaginary and perpendicular to $n_i$ and $n_i^\prime$ respectively, there is a unique $\phi_i$ satisfying the following conditions:
\begin{align*}
    \phi_i^{-1}\cdot n_i \cdot \phi_i &=n_i^\prime,\\
    \overline{\phi_i}\cdot \mathrm{cr}_\mathfrak{X}(\gamma_o) \cdot \phi_i &=\mathrm{cr}_{\mathfrak{X}^\prime}(\gamma_o).
\end{align*}
Fixing this $\phi_i$, the values of the other $\phi$ on the loop $\phi_o$ are then all compatibly determined. \\
To determine the values of $\phi$ on the other vertices $j$ away from $\gamma_o$ we just need to again take some path between $i$ and $j$, if the path has even length, we are done. Otherwise we can precompose the path with $\gamma_o$ and obtain a even path. It remains to determine the values of $\phi$ on this path by \eqref{eqn:induction}.
\end{proof}

\section{The smooth intrinsic Dirac operator}
\label{sec:relation}
In this section we are going to describe the exact connection between the extrinsic and intrinsic Dirac operators (for a more detailed treatment of spin structures and Dirac operators see \cite{lawson_spin_1990}). The notation $\Gamma(P)$ stands for the space of sections of some fiber bundle $P$.

Again we start with the smooth setup: Suppose $X$ is an oriented surface and $f:X\rightarrow \R^3$ is an immersion. Let $Cl_3\rightarrow \R^3$ be the trivial Clifford bundle over $\R^3$ and let $\mathcal{S}_{\mathbb{R}^3}\rightarrow \R^3$ be the corresponding trivial spinor bundle. These two bundles both can be pulled back to $X$ through the map $f$: $Cl_3|_X=f^*(Cl_3)$ and $\mathcal{S}_{\mathbb{R}^3}|_X=f^*(\mathcal{S}_{\mathbb{R}^3})$. Furthermore since there is a natural identification $Cl_2\hookrightarrow Cl_3^{even}$ by $v\mapsto n\cdot v$ where $n$ is the normal of $X$ in $\R^3$, we can define the Clifford representation
\begin{align}
\label{eqn:mul_id}    \rho: Cl_2 & \rightarrow \mathrm{End}(\mathcal{S}) \\
    v & \mapsto \rho_3(n\cdot v)\nonumber
\end{align}
where $\rho_3$ is the Clifford representation of $Cl_3$.\\
Suppose $\phi$ is a section of the spinor bundle, i.e., $\phi \in \Gamma(\mathcal{S}_X)$,  the Dirac operator is
\begin{align*}
    D: \Gamma(\mathcal{S}_X)  & \rightarrow \Gamma(\mathcal{S}_X)\\
    \phi & \mapsto \rho(e_1)\cdot \nabla_{e_1}\phi + \rho(e_2)\cdot \nabla_{e_2}\phi
\end{align*}
where $\{e_1,e_2\}$ is an oriented orthonormal frame of $X$ and $\nabla$ is the spin connection of $X$.\\
Let $c\in \Gamma(P_{\mathrm{Spin}}(\mathbb{R}^3))$ be a global parallel section of the spin bundle. Since
\begin{itemize}
\item The intrinsic spinor bundle $\mathcal{S}_X$ can be identified with the trivial ambient spinor bundle $\mathcal{S}_{\mathbb{R}^3}$ by \eqref{eqn:mul_id}.
\item Any section of the spinor bundle $\mathcal{S}_{\mathbb{R}^3}$ can be represented by a pair $(c,\phi_c)$, where $\phi_c\in C^\infty(\mathbb{R}^3,\mathbb{H})$, because $\mathcal{S}_{\mathbb{R}^3}$ is defined as an associated bundle $\mathcal{S}_{\mathbb{R}^3} := P_{\mathrm{Spin}}(\mathbb{R}^3) \times_\sim \mathbb{H}$, where $\sim$ is given in \eqref{eqn:equiv}.
\end{itemize}
$c$ induces an isomorphism:
\begin{align*}
\mathfrak{c}: \Gamma(\mathcal{S}_X) &\cong \Gamma_X(\mathcal{S}_{\mathbb{R}^3})  \rightarrow C^\infty(X,\H)\\
\phi & \mapsto \quad (c,\phi_c)\mapsto \phi_c.
\end{align*}
\begin{theorem}
\label{thm:relation}
Let $f:X\hookrightarrow \R^3$ be an isometric surface immersion. Then we have
\[\mathfrak{c} \circ (D-H) \circ \mathfrak{c}^{-1}=D_f,\]
where $D_f$ is the extrinsic Dirac operator \eqref{eqn:ex_dirac} and $H$ is the mean curvature of $f$.
\end{theorem}
\begin{proof}
 Note that the covariant derivative of the ambient space and its hypersurface differ by a second fundamental form (see \cite{hijazi_dirac_2001})
\begin{align*}
\nabla_X Y &= \tilde{\nabla}_XY-\langle \tilde{\nabla}_X Y,n\rangle n\\
&= \tilde{\nabla}_XY+\langle Y,\tilde{\nabla}_Xn\rangle n\\
&= \tilde{\nabla}_X Y- \mathrm{II}(X,Y)n
\end{align*}
and the corresponding spinor connection satisfies \begin{align*}
\nabla_{X} \phi =\tilde{\nabla}_X \phi -\frac{1}{2} \mathrm{II}(e_1,X)e_1\cdot n\cdot \phi-\frac{1}{2}\mathrm{II}(e_2,X)e_2\cdot n\cdot \phi\; .
\end{align*}
It yields
\begin{align}
D\phi &= \rho(e_1)\cdot \nabla_{e_1} \phi+\rho(e_2)\cdot \nabla_{e_2} \nonumber\\
&= \rho_{3}(n)\cdot \rho_{3}(e_1)\cdot \Bigg(\tilde{\nabla}_{e_1}  \phi-\frac{1}{2}\mathrm{II}(e_1,e_1)\rho_3(e_1)\cdot \rho_3(n)\cdot \phi \nonumber \\
&-\frac{1}{2}\mathrm{II}(e_1,e_2)\rho_3(e_2)\cdot \rho_3(n)\cdot \phi \Bigg) +\rho_{3}(n) \cdot\rho_{3}(e_2)\cdot  \Bigg(\tilde{\nabla}_{e_2}\phi \label{eqn:dirac1}\\
&-\frac{1}{2}\mathrm{II}(e_2,e_1)\rho_3(e_1)\cdot \rho_3(n)\cdot \phi -\frac{1}{2}\mathrm{II}(e_2,e_2)\rho_3(e_2)\cdot \rho_3(n)\cdot \phi \Bigg) \nonumber\\
&=\rho_3(N)\cdot \rho_3(e_1)\cdot \tilde{\nabla}_{e_1}\phi+\rho_3(n)\cdot\rho_3(e_2)\cdot  \tilde{\nabla}_{e_2}\phi+H\phi\nonumber
\end{align}
where $\tilde{\nabla}$ is the Levi-Civita connection of $\R^3$.

Now let us take the global parallel frame $c$ with the following identifications
\[e_1 \mapsto \dif f(e_1),\quad e_2\mapsto \dif f(e_2),\quad n\mapsto N \]
where $\dif f(e_1)$, $\dif f(e_2)$, and $N$ are imaginary quaternions.  Since $c$ is parallel, the covariant derivative reduces to the partial derivative $\partial$. Hence \eqref{eqn:dirac1} becomes:
\begin{align}
    \mathfrak{c}\circ (D-H)\circ  \mathfrak{c}^{-1} &=N \cdot  \dif f (e_1) \cdot \partial_{e_1}  +N \cdot \dif f (e_2) \cdot \partial_{e_2} \nonumber \\
    \label{eqn:dirac2}
    &=  \dif f(e_2) \partial_{e_1} - \dif f (e_1)  \partial_{e_2} .
\end{align}
On the other hand we have (see \cite{chubelaschwili_variational_2016} for more details)
\begin{align}
    D_f  &= -\frac{\dif f \wedge \dif }{\lvert \dif f\rvert^2} \nonumber\\
    &= -\frac{(\dif f(e_1) e_1^* +\dif f(e_2)e_2^*)\wedge (e_1^*\partial_{e_1}  + e_2^*\partial_{e_2} )}{\lvert \dif f \rvert^2} \nonumber\\
    &=-\frac{\left(\dif f(e_1)\partial_{e_2}-\dif f(e_2)\partial_{e_1}\right) e_1^*\wedge e_2^*}{\lvert \dif f \rvert^2} \nonumber \\
    \label{eqn:dirac3}
     &= -\dif f(e_1)\partial_{e_2}+\dif f(e_2)\partial_{e_1} .
\end{align}
Comparing \eqref{eqn:dirac2} with \eqref{eqn:dirac3} we finally find
\[\mathfrak{c}\circ (D-H) \circ \mathfrak{c}^{-1} =D_f\]
\end{proof}
\section{A Discretization of the intrinsic Dirac operator}
\label{sec:intrinsic}
Next, we aim to find a discrete version of the above relation. We start with
\subsection{A Discrete principal bundle}
Following the ideas from \cite{bobenko_complex_2016}  we construct the discrete principal bundle by the connection between neighbouring faces.
\begin{definition}
Let $X$ be an oriented net. We call $(P,X,G,\eta)$ a discrete principal bundle with connection if %
\begin{enumerate}
\item 
  each face $\Delta_i$ is assigned with a manifold $P_i$ with a right action, free and transitive, by a Lie group $G$.
\item $P=\{P_i\}$ is a collection of the manifolds $P_i$.
\item each oriented dual edge $\vec{ij}$ is endowed with a connection $\eta_{ij}:P_i \rightarrow P_j$ such that $\eta_{ij}(p\cdot g)= \eta_{ij}(p)\cdot g$ and  $\eta_{ji}\circ \eta_{ij}=\mathrm{Id}$.
\end{enumerate}
\end{definition}

Integrating the connections along the fundamental loop around a vertex $v$ we obtain the holonomy $\Omega_p^v\in G$:
 \[p\cdot \Omega^{v}_p :=\eta_{n,1}\circ\ldots\circ\eta_{23}\circ \eta_{12}(p)\]
  It is easy to see that $\Omega^{v}_{pg}=\mathrm{Ad}_{g^{-1}}\Omega^{v}_p$, hence the holonomy of the same fibre all lie in the same conjugate class. \\
We know that the spin group $\mathrm{Spin}(n)$ is a two-fold covering of $\mathrm{SO}(n)$, namely the following short exact sequence holds:
\[
0\rightarrow \mathbb{Z}_2 \rightarrow \mathrm{Spin}(n)\xrightarrow{\xi_0} \mathrm{SO}(n)\rightarrow  0\]
where $\xi_0$ is the adjoint representation.
Given a $\mathrm{SO}(n)$-principal bundle \[(P_{SO},X,\mathrm{SO}(n),\eta),\] a lifting is a $\mathrm{Spin}(n)$-principal bundle $(P_{Spin},X,\mathrm{Spin}(n),\tilde{\eta})$ together with a set of maps $\xi_i: P_{Spin}^i\rightarrow P_{SO}^i$ which are compatible with the connections, i.e. the following diagram commutes at each dual edge $\vec{ij}$:
\[\begin{tikzcd}
P_{Spin}^i\arrow{r}{\eta^{ij}} \arrow{d}{\xi} & P_{Spin}^j \arrow{d}{\xi}\\
P_{SO}^i\arrow{r}{\tilde{\eta}^{ij}} & P_{SO}^j
\end{tikzcd}\]
If $n=2$, then since $\mathrm{SO}(2)$ and $\mathrm{Spin}(2)$ are both abelian groups, the holonomy of the loop is well-defined without specifying an point $p$ in the fibre.

\subsection{Discrete associated bundle and Clifford multiplication}
\begin{definition}
We consider a principal $G$-bundle $P_G$ and a vector space $W$ with the left action by $G$. Take a product space $P_G\times W$ modulo the relation $\sim$:
\begin{equation}
\label{eqn:equiv}
(p,v) \sim (pg^{-1},gv)
\end{equation}
We call $P_G\times_\sim W$ the associated bundle to $P_G$. The connection on the associated bundle is 
\[(p,v)_i\mapsto (\eta_{ij}(p),v)_j.\]
\end{definition}
Since 
\[\begin{tikzcd}
(p,v)_i \arrow[mapsto]{r} \arrow{d}{\sim}& (\eta_{ij}(p),v)_j \arrow{d}{\sim}\\
(p\cdot g^{-1},gv) \arrow[mapsto]{r} & (\eta_{ij}( p) \cdot g^{-1}, gv)
\end{tikzcd}\]commutes, the connection is well-defined on the associated bundle. 
In order to define the Clifford multiplication on bundle level we need to check the covariance. Let $S$ denote the irreducible Clifford module. Since there is a bundle isomorphism $P_{SO}\times W\cong P_{\spin}\times_{\mathrm{Ad}} W$, the Clifford multiplication can be defined as follows 

\begin{align*}
  \left(P_{\spin} \times_{\mathrm{Ad}} W \right)  \times \left(P_{\spin}\times S\right)  &\rightarrow %
   P_{\spin}\times S \\ 
  (p,v)  \times (p,x) &\mapsto%
   (p,v\cdot x)
\end{align*}
If we change $p$ to $pg^{-1}$, it yields \[(pg^{-1},gvg^{-1})\cdot (pg^{-1},gx)=(pg^{-1},gv\cdot x)=(p,vx).\]
Hence, the multiplication is independent of the choice of $p$. It is also easy to see that the Clifford multiplication is compatible with the connection, i.e. $\eta_{ij}(v)\cdot \tilde{\eta}_{ij}(x)=\tilde{\eta}_{ij}(v\cdot x)$, or
\[
\begin{tikzcd}
(p,v)_i\times(p,x)_i \arrow{d} \arrow{r}{\tilde{\eta}_{ij}} & (\tilde{\eta}_{ij}(p),v)_j\times (\tilde{\eta}_{ij}(p),x)_j \arrow{d}\\
(p,v\cdot x)_i\arrow{r}{\tilde{\eta}_{ij}} & (\tilde{\eta}_{ij}(p),v\cdot x)_j
\end{tikzcd}\]
commutes.
\subsection{The Discrete Dirac operator}
In order to introduce a discrete version of the spinor connection, which is necessary for the intrinsic Dirac operator, we propose the following setting of discrete intrinsic nets, which mimics smooth surfaces with Riemannian metric. In the end we will show that the discrete intrinsic Dirac operator arising from this setting couples with discrete extrinsic Dirac operator introduced in \cref{sec:extrinsic} very well. Therefore they form a consistent framework together with the face edge-contraint net setting in \cref{sec:extrinsic}. The notion of discrete spinor connection is compatible with the one in the recent work \cite{chern_2018}, which is used for shape embedding problems}.

\begin{definition}
\label{def:intrinsic_net}
An intrinsic net is an oriented net such that each face $\Delta_i$ is endowed with an Euclidean affine plane $\mathrm{Affine}(\Delta_i)$ and every oriented edge $e_{ij}$ in $\Delta_i$ is identified with a tangent vector (a vector attached to a point) in $\mathrm{Affine}(\Delta_i)$, denoted by $e_{ij}^i$ such that 
\begin{itemize}
\item the common edge is identified with the same length in the neighbouring faces, i.e., $\lvert e_{ij}^i \rvert =\lvert e_{ij}^j\rvert$,
\item for each face $\Delta_i$ the extension lines of the tangent vectors $e_{ij}^i$ form a convex polygon with counterclockwise orientation.
\end{itemize}
\end{definition}
\begin{figure}[h]
\centering
\def\svgwidth{0.5\columnwidth}
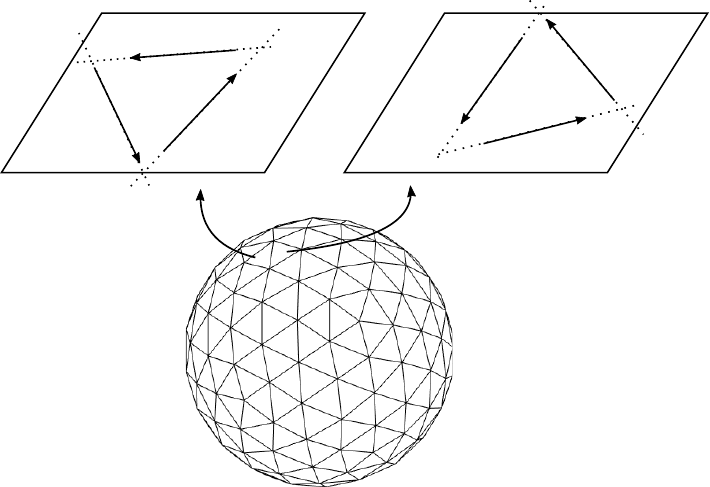
\caption{The intrinsic net}\label{fig:intrinsicnet}
\end{figure}
\begin{remark}
The the edges in a face do {\em not} need to form a closed polygon. However, it makes sense to define the angle between any pair of edges in a face by taking the angle between their extension lines (see \cref{fig:intrinsicnet}).
\end{remark}

\begin{definition}
  An oriented orthonormal frame of a face $\Delta_i$ is an oriented affine isometric map 
  \[p^i: \R^2\rightarrow \mathrm{Affine}(\Delta_i).\]
\end{definition}
Let $p_1^i:=p^i(\begin{pmatrix} 1 \\0\end{pmatrix})$ and $p_2^i:=p^i(\begin{pmatrix} 0 \\ 1 \end{pmatrix})$.
Given a frame at $\Delta_i$, the vector $e_{ij}^i$ can be represented by a linear combination of that frame, denoted by $p^i(e_{ij}^i)$ or $\mathfrak{e}_{ij}^i$. 
\begin{definition}
Suppose $X$ is an intrinsic net. An orthonormal frame bundle with Levi-Civita-connection $P^{LC}_{\mathrm{SO}}\rightarrow X$ is a $\mathrm{SO}(2)$-bundle consisting of all the orthonormal frames at each face $\Delta_i$ satisfying $\big(\eta_{ij}(p^i)\big)(e_{ij}^j)=p^i(e_{ij}^i)$.
\end{definition}
Now one can take any lift of the principal bundle with Levi-Civita-connection $P_{\mathrm{Spin}}^{LC}\rightarrow P^{LC}_{\mathrm{SO}}$. Then the tangent bundle can be constructed by
\[\mathrm{T}X:= P_\mathrm{Spin}^{LC}\times_\mathrm{Ad} \R^2\]
and the spinor bundle can be constructed by
\[\mathcal{S} =P_{\mathrm{Spin}}^{LC}\times_\mathrm{L} S\]
where $S\cong\H$ is the irreducible Clifford module of $\mathrm{Spin}(2)$ and $\mathrm{L}$ denotes the left action of $\mathrm{Spin}(2)$ on $S$. Note that there is an isomorphism 
\begin{align*}
\mathrm{T}X_i &\xrightarrow{\cong} \mathrm{Affine}(\Delta_i), \\
(e,v) &\mapsto e(v).
\end{align*}
Therefore the Clifford multiplication is defined by
\begin{align*}
    \mathrm{Affine}(\Delta_i) \times \mathcal{S}_i &\rightarrow \mathcal{S}_i \\
    (e,v) \times (e,x)  &\mapsto (e,v\cdot x)
\end{align*}
and with this we are finally able to formulate a discrete intrinsic Dirac operator as follows:
\begin{definition}[Discrete Dirac operator]
Given an intrinsic net $X$ and the principal bundle $P_{\spin}\rightarrow P_{\so}$ over $X$. The Dirac operator $D$ is a map $\Gamma(\mathcal{S}) \rightarrow \Gamma(\mathcal{S})$, where $\Gamma(\mathcal{S})$ is the sections of $\mathcal{S}$, defined as follows:

\[D(\phi)_{i} = \frac{1}{2}\sum\limits_{j} e_{ij}\cdot \tilde{\eta}_{ji}(\phi_j).\]
\end{definition}
Note, that there is a well-defined Hermitian product
\begin{align*}
    \Gamma(\mathcal{S}) \times \Gamma(\mathcal{S}) & \rightarrow \mathcal{H}, \\
    \langle (p,x_1), (p,x_2)\rangle  & =\overline{x_1}\cdot x_2\; .
\end{align*}
\begin{theorem}
Any $\phi$ satisfying the Dirac equation
\[D\phi=\bm{\rho} \phi\]
where $\bm{\rho}: F \rightarrow \R$ is a real-valued function, gives rise to a face-edge-constraint net by:
\begin{align*}
E_{ij}& = \langle \phi_i, e_{ij}\cdot \tilde{\eta}_{ji}(\phi_j)\rangle, \\
n_i &= \frac{1}{\lvert \phi_i\rvert^2}\langle \phi_i, \qk \cdot \phi_i\rangle.
\end{align*}
\end{theorem}
\begin{proof}Compute
\begin{align*}
\sum\limits_j E_{ij} &=\sum\limits_j\langle \phi_i,e_{ij}\cdot \tilde{\eta}_{ji}(\phi_j)\rangle \\
&=\langle \phi_i,2(D\phi)_i\rangle\\
&=2\langle \phi_i,\bm{\rho} \phi_i\rangle\\
&=2\bm{\rho} \lvert \phi\rvert^2
\end{align*}
which is a real number.
\end{proof}
We will call these a face-edge-constraint realization of the underlying intrinsic net with respect to the spinor $\phi$.
\subsection{Explicit construction of the intrinsic Dirac operator and face-edge-constraint realizations}
Now let us derive an explicit formula for the Dirac equation as well as the face-edge-constraint realizations. We begin by choosing an orthonormal frame $p_i=(p_1^i,p_2^i)$ at each face.\\
Let $g_{ij}\in \mathrm{Spin}(2)$ be defined by $p_i\cdot g_{ij}=\tilde{\eta}_{ji}( p_j)$. Since $\tilde{\eta}_{ij}\circ\tilde{\eta}_{ji}=\mathrm{Id}$, we have $g_{ij}=g_{ji}^{-1}$. Then we take an isometric embedding of the the affine plane $\mathrm{Affine}(\Delta_i)$ and $\mathrm{Affine}(\Delta_j)$ into $\qi$-$\qj$-plane such that
\begin{enumerate}
\item the common edge $e_{ij}^i$ and $e_{ij}^j$ coincide in this embedding.
\item $p_1^i$ is mapped to $\qi$ and $p_2^i$ is mapped to $\qj$.
\end{enumerate}
Now, every vector in these two affine planes can be identified with a quaternion in the $\qi$-$\qj$-plane by:
\[v=xp_1^i+yp_2^j \mapsto x\qi+y\qj.\]
In particular
\begin{align*}
    p_1^j &\mapsto c_{11} \qi+ c_{12} \qj, \\
    p_2^j &\mapsto c_{21} \qi + c_{22} \qj.
\end{align*}
We can find a quaternion $g_{ji}$ such that 
\begin{align*}
    c_{11}\qi + c_{12}\qj &= g_{ij} \qi g_{ij}^{-1}, \\
    c_{21}\qi + c_{22}\qj &= g_{ij} \qj g_{ij}^{-1}.
\end{align*}
In fact $g_{ij}$ is uniquely defined up to a sign, which represent different liftings of the connection. We will see in the next section that the choice of the lifting actually determines the spin multi-ratio. \\
The parallel transport from a neighbouring face $\Delta_j$ is:
\begin{align*}
\tilde{\eta}_{ji}((p_j,\phi_j)) &=(p_i\cdot g_{ij},\phi_j)\\
&=(p_i,g_{ij}\cdot \phi_j).
\end{align*}

In $\mathrm{Affine}(\Delta_i)$ we can write
\[e_{ij}^i=x p^i_1+yp^i_2\mapsto \mathfrak{e}_{ij}^i=x \qi+y\qj.\]
Therefore, the Dirac operator becomes
\begin{align*}
D(\phi)_i &= \frac{1}{2}\sum\limits_j \mathfrak{e}_{ij}^i\cdot g_{ij}\cdot \phi_j
\end{align*}
and  in the local frame $p_i$ the Dirac equation has the form
\[\frac{1}{2}\sum\limits_j \mathfrak{e}_{ij}^i\cdot g_{ij}\cdot \phi_j=\bm{\rho}_i \phi_i.\]
Moreover, a face-edge-constraint realization is given by the explicit formula
\begin{align}
\label{eqn:edgemap}
E_{ij} & =\overline{\phi_i}\cdot \mathfrak{e}_{ij}^i\cdot g_{ij}\cdot \phi_j, \\
n_i &= \phi_i^{-1} \cdot \qk \cdot \phi_i.
\end{align}
To see that this realization is well-defined, we first compute
\begin{align}
\label{ean:edgemap2}
E_{ji} &=\overline{\phi_j}\cdot  \mathfrak{e}_{ji}^j\cdot g_{ji}\cdot  \phi_i.\nonumber
\end{align}
Note that $\mathfrak{e}_{ji}^j=-\mathfrak{e}_{ij}^j=-g_{ij}^{-1}\cdot \mathfrak{e}_{ij}^i \cdot g_{ij}$ and $g_{ij}=g_{ji}^{-1}$. This implies
\begin{align} E_{ji} &= -\overline{\phi_j} \cdot g^{-1}_{ij} \mathfrak{e}_{ij}^i \cdot g_{ij} \cdot g_{ji}\cdot \phi_i \\
&= -\overline{\phi_j}\cdot g_{ij}^{-1}\cdot  \mathfrak{e}_{ij}^i\cdot \phi_i
\end{align}
and by $\overline{g_{ij}^{-1}}=g_{ij}$ we obtain $E_{ij}=\overline{E_{ji}}$. Finally we need to show that
\[E_{ij}^{-1}\cdot n_i \cdot  E_{ij} = - n_j.\]
By direct computation we see
\begin{align}
E_{ij}^{-1}\cdot n_i \cdot E_{ij} & = \left(\overline{\phi_i}\cdot \mathfrak{e}_{ij}^i\cdot g_{ij}\cdot \phi_j\right)^{-1} \cdot \phi^{-1}_i\cdot \qk \cdot \phi_i\cdot \overline{\phi_i}\cdot \mathfrak{e}_{ij}^i\cdot g_{ij}\phi_j \\
&= \phi_j^{-1} g_{ij}^{-1} \cdot (-\mathfrak{e}_{ij}^i) \cdot \qk \cdot \mathfrak{e}_{ij}^i\cdot g_{ij}\cdot \phi_j.
\end{align}
Since $\mathfrak{e}_{ij}^i$ lies in the $\qi$-$\qj$-plane,
\[-\mathfrak{e}_{ij}^i\cdot \qk \cdot \mathfrak{e}_{ij}^i=-\qk\]
and $g_{ij}$ has the axis parallel to $\qk$, so  $g_{ij}^{-1}\cdot \qk \cdot g_{ij}=\qk$, it follows that
\[E_{ij}^{-1}\cdot n_i\cdot E_{ij} = -\phi_{j}^{-1}\cdot \qk \cdot \phi_j = -n_j.\]
\subsection{The Preferred Choice for the Lifting}
We know that in an intrinsic net each edge admits two liftings with opposite sign, hence an intrinsic net with $n$ edges have $2^n$ different spinor connections. Now we are going to show that among all these spinor connections there are some more reasonable ones, called the preferred liftings, which correspond to the spinor structures in the smooth case.\\
Similar to \cref{def:regular} we call a vertex in an intrinsic net regular if and only if
\[ \langle \mathfrak{e}_{i-1,i}^i \times \mathfrak{e}_{i,i+1}^i, \qk\rangle >0.\]
Let $v$ be a regular vertex with even degree and $\mathfrak{X}$ be the face-edge-constraint realization of $(X,\mathcal{A})$ with respect to the spinor $\phi$, then
\begin{align*}
  \mathrm{cr}_\mathfrak{X}(v) =\; & \overline{E_{12}}^{-1}\cdot E_{23} \cdots \overline{E_{n-1,n}}^{-1}\cdot E_{n,1} \\
  =\; &\phi_{1}^{-1}\mathfrak{e}_{12}^1\cdot g_{12}\frac{\phi_2}{\lvert \phi_2\rvert^2} \cdot \overline{\phi_2} \cdot \mathfrak{e}_{23}^2 \cdot g_{23}\cdot \phi_3 \cdots\\
  &\cdots \phi_{n-1}^{-1} \mathfrak{e}_{n-1,n}^{n-1}\cdot g_{n-1,n}\frac{\phi_{n}}{\lvert \phi_n\rvert^2} \cdot \overline{\phi_{n}}\cdot  \mathfrak{e}_{n,1}^n\cdot g_{n,1}  \cdot \phi_1\\
=\; & \phi_1^{-1}\cdot \mathfrak{e}_{12}^{1}\cdot g_{12}\cdot \mathfrak{e}_{23}^2\cdot g_{23}\cdots \mathfrak{e}_{n-1,n}^{n-1}\cdot g_{n-1,n}\cdot \mathfrak{e}_{n,1}^n\cdot  g_{n,1}\cdot \phi_1 \\
 =\; & \phi_1^{-1} \cdot  \mathfrak{e}_{12}^1 \cdot  (g_{12}\mathfrak{e}_{23}^2g_{12}^{-1}) \cdot g_{12}\cdot g_{23} \cdot \mathfrak{e}_{34}^3 \cdots \mathfrak{e}_{n-1,n}^{n-1}\cdot g_{n-1,n}\cdot \mathfrak{e}_{n,1}^n\cdot g_{n,1}\phi_{1} \\
  =\; &\phi_1^{-1} \cdot \mathfrak{e}_{12}^1\cdot (g_{12}\mathfrak{e}_{23}^2g_{12}^{-1}) \cdot  (g_{12}g_{23}\mathfrak{e}_{34}^3 g_{23}^{-1}g_{12}^{-1})\cdots\\
  &\cdots (g_{12}\ldots g_{n-1,n} \mathfrak{e}_{n,1}^ng_{n-1,n}^{-1}\cdots g_{12}^{-1})\cdot (g_{12}\cdot \cdots\cdot g_{n,1})\cdot \phi_1 \\
=\; &\phi_1^{-1}\cdot \mathfrak{e}_{12}^1\cdot \mathfrak{e}_{23}^1\cdot \cdots \mathfrak{e}_{n-1,n}^1\cdot \mathfrak{e}_{n,1}^1\cdot(g_{12}\cdot \cdots \cdot g_{n,1}) \cdot \phi_1\;.
\end{align*}
We call $\mathfrak{X}=(X,f,n)$ a classical realization of $(X,\mathcal{A})$ if and only if $\mathfrak{X}$ is classical and all the internal angles are preserved:
\[\angle (df_{i-1,i},df_{i,i+1})=\angle (\mathfrak{e}_{i-1,i}^i,\mathfrak{e}_{i,i+1}^i).\]
For a classical realization, observe that $e_{12}^1\cdot e_{23}^1\cdot \cdots e_{n-1,n}^1\cdot e_{n,1}^1$ actually coincides with the edge part of the spin multi-ratio of the classical realization. Hence $g_{12}\cdot \cdots \cdot g_{n,1}$ should coincide with the curvature part of the spin multi-ratio.
\begin{definition}
Let $(X,\mathcal{A})$ be an intrinsic net with only regular vertices. A choice of lifting is called a preferred lifting if 
\[g_{12}\cdot \cdots\cdot g_{n,1}=
\cos\frac{K(v)}{2} + \sin\frac{K(v)}{2} \qk\]
holds for all vertices.
\end{definition}

\begin{lemma}[A Gauss-Bonnet theorem for intrinsic nets]
\label{lem:gaussbonet}
Let $(X,\mathcal{A})$ be an intrinsic net (\cref{def:intrinsic_net}). Suppose the total angular defect $\sum_i K(v)$ is the sum of the angular defects of all the vertices. Then we have
\[\sum\limits_{vertices} K(v)=2\pi\chi\]
where $\chi$ is the Euler characteristic. 
\end{lemma}
\begin{proof}
We have
\begin{align*}
\sum\limits_{vertices} K(v) & =\sum\limits_{vertices}\left(2\pi-\Sigma(v)\right)\\
&= 2\pi \lvert V\rvert -\Sigma\\
&=2\pi\lvert V \vert -\sum\limits_{faces} \Sigma(\Delta_i)
\end{align*}
where $\Sigma(v)$ is the sum of the interior angles at the vertex $v$, $\Sigma(\Delta_i)$ is the sum of the interior angles in $\Delta_i$ and $\Sigma$ is the sum of all the interior angles. Assuming that in the face $\Delta_i$ the extension lines of the vectors form an oriented convex  $s_i$-sided polygon, then the sum of the interior angles is $(s_i-2) \pi$ and
\[ \Sigma(\Delta_i)= (s_i-2)\pi.\]
Further note, that $\sum\limits_{faces} s_i = 2\lvert E\rvert$, hence 
\begin{align*}
    \sum\limits_{vertices} K(v) &= 2\pi\lvert V\rvert -\sum\limits_{faces} (s_i-2)\pi \\
    &= 2\pi \lvert V\rvert - 2\pi\lvert E\rvert +2\pi \lvert F\rvert \\
    &= 2 \pi \chi. 
\end{align*}
\end{proof}
\begin{theorem}
\label{thm:spinexisten}
Every intrinsic net $(X,\mathcal{A})$ (\cref{def:intrinsic_net}) has a preferred lifting.
\end{theorem}
\begin{proof}
Any choice of the lifting $g_{ij}$ gives a $2$-cochain $\sigma$ in the following way. Let $\mu$ be a map from the vertices to $\mathrm{Spin}(2)$ defined by
\[\mu[v]= g_{12}\cdot \cdots \cdot g_{n,1}\]
and let $\nu$ be the map defined by
\[\nu[v]= \cos\frac{K(v)}{2} + \sin\frac{K(v)}{2} \qk.\]
Since $g_{i,i+1}$ all lie in the $\qi$-$\qj$-plane, $\mu$ and $\nu$ both indeed have the codomain $\mathrm{Spin}(2)$. Since $\mathrm{Spin}(2)$ is abelian, $\mu$ and $\nu$ can be linearly extended to the $2$-cochains of $X^*$, i.e., 
\[\mu,\nu \in \mathrm{C}^2(X^*,\mathrm{Spin(2)}).\]
The $2$-cochain $\sigma$ is defined by
\[\sigma[v]:=\mu[v]\cdot \nu[v]^{-1}.\]
Since $g_{12}\cdot \cdots \cdot g_{n,1}=\pm(\cos\frac{K(v)}{2} + \sin\frac{K(v)}{2} \qk)$, $\sigma$ is actually a $2$-cochain with coefficient $\mathbb{Z}_2$, i.e.,
\[\sigma\in \mathrm{C}^2(X^*,\mathbb{Z}_2).\]
Clearly $\sigma$ takes the value:
\[\sigma[v] =\begin{cases} 1 & g_{21}\cdot \cdots \cdot g_{n,1} = \cos\frac{K(v)}{2} + \sin\frac{K(v)}{2} \qk \\ -1 & g_{21}\cdot \cdots \cdot g_{n,1}=-\cos\frac{K(v)}{2} - \sin\frac{K(v)}{2}\qk\end{cases}.\]
If $g_{ji}$ is a preferred lifting then $\sigma=0$. If we change the lifting at some edges, then it leads to a $2$-cochain $\sigma^\prime$ which only differs from $\sigma$ by a differential of a $1$-cochain:
\[\sigma^\prime =\sigma +\dif \delta \]
where $\delta\in \mathrm{C}^1(X^*,\mathbb{Z}_2)$. It implies that even though $\sigma$ as a cochain depends on the lifting $g_{ji}$,
\[\bar{\sigma}\in \mathrm{H}^2(X^*,\mathbb{Z}_2)\]
as a cohomology class doesn't depend on the choice of the lifting but only depends on the $\mathrm{SO}$-connection. Moreover $\bar{\sigma}=0$ if and only if there exists a preferred lifting. Observe that
\[\sigma[X^*]=\mu[X^*]\cdot \nu[X^*]^{-1}\]
and we have $\mu[X^*]=\mathrm{Id}$ because every $g_{ji}$ and $g_{ij}$ always appear in pair in $X^*$. Furthermore $\sum\limits_{v\in V}K(v)=\chi\cdot 2\pi$ by \cref{lem:gaussbonet}, which is always an even number for a oriented surface. Hence $\nu[X^*]=\cos\frac{\chi}{2} + \sin\frac{\chi}{2}\qk=\mathrm{Id}$ and then 
\[\sigma [X^*]=1.\]
We know that there is only one nontrivial class $\omega\in \mathrm{H}^2(X^*,\mathbb{Z}_2)$ but $\omega[X^*]=-1$, thus $\omega\neq \bar{\sigma}$ and $\bar{\sigma}=0$.
\end{proof}

\begin{definition}
Given an intrinsic net $(X,\mathcal{A})$ satisfying the condition in \cref{lem:gaussbonet}, the spin equivalence class is the set of the  pairs $(X,\mathcal{A},\tilde{\eta})$ where $\tilde{\eta}$ is a preferred lifting of $(X,\mathcal{A})$ modulo the spin equivalence relation. 
\end{definition}
\begin{theorem}
\label{thm:spinnumber}
The spin equivalence class of an intrinsic net with Betti number $b$ has $2^{b}$ elements. 
\end{theorem}
\begin{proof}
Let $(X,\mathcal{A},\eta)$ and $(X,\mathcal{A},\eta^\prime)$ be two preferred liftings of the same underlying intrinsic net. Since the spin multi-ratio at each vertex $v$ should be the same for two liftings, at each vertex there should be even numbers of incident edges $e_{ij}$ such that the $\eta_{ij}$ and $\eta^\prime_{ij}$ have reversed signs. Hence all these edges form some closed boundaries. \\
For a simply-connected net these boundaries would create some separated disk-like areas. It's easy to see that any loops always cross these boundaries with an even number of times. Therefore the spin multi-ratio for all the even loops are the same for the liftings $\eta$ and $\eta^\prime$, meaning that they are spin equivalent.\\
Suppose the $X$ has the Betti number $b$, we can always find $2b$ closed curves which represent  different non-trivial homology classes. Pick any of such a closed curve, flip the signs of the spinor connections all along this curve and we obtain a new spin equivalence class.    
\end{proof}
\begin{remark}
Recall that in the smooth theories, an oriented manifold has the spin structure if and only if the second Stiefel-Whitney class is zero. Hence a oriented surface is spin if and only if the Euler characteristic is even (which is true for all oriented surfaces). Furthermore a spin manifold has $2^{2b}$ number of spin structures.  Clearly \cref{thm:spinexisten} and \cref{thm:spinnumber} show that our discretization preserves all these results.

\end{remark}
\section{The connection between the extrinsic and intrinsic Dirac operators in the discrete case}
In the last section we started with an intrinsic net and constructed face-edge-constraint realizations by solving the Dirac equation. Now we are going to discuss the question: How can we construct the intrinsic net from a given face-edge-constraint net? In fact we will see that each face-edge-constraint net is associated with an intrinsic  net and a constant spinor field $\phi_c$ with unit length in the ambient space $\R^3$ induces a spinor field on the intrinsic net. With this induced spinor field one can reconstruct the original egde-constraint net from the associated Riemannain net. Moreover the relation between the extrinsic and intrinsic operators still holds in the discrete case. Precisely, the ideas can be depicted as follows:
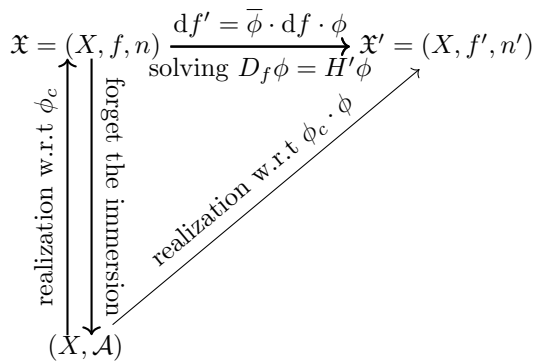
\begin{figure}[H]
    \centering
    \begin{tikzpicture}[scale=0.80]
    \node (ex1) at (0,0) {$\mathfrak{X}=(X,f,n)$};
    \node (ex2) at (6,0) {
    $\mathfrak{X}^\prime=(X,f^\prime,n^\prime)$};
    \node (in) at (0,-5) {$(X,\mathcal{A})$};
    \draw[->,thick] (ex1) -- (ex2) node[above,midway] {$\dif f^\prime =\overline{\phi}\cdot \dif f \cdot\phi$};
    \draw[->,thick] (ex1) -- (ex2) node[below,midway] {
    solving $D_f \phi=H^\prime \phi$};
    \path [->,every node/.style={sloped,anchor=south,auto=false}] (in) edge node {realization w.r.t  $\phi_c\cdot \phi$} (ex2);
    \path [->,thick,every node/.style={sloped,anchor=south,auto=false}] ($(in)-(0.3,-0.2)$) edge  node {realization w.r.t $\phi_c$} ($(ex1)-(0.3,0.2)$);
    \path [->,thick,every node/.style={sloped,anchor=south,auto=false}] ($(ex1)+(0.1,-0.2)$) edge  node {forget the immersion} ($(in)+(0.1,+0.2)$);
    
    \end{tikzpicture}
    
    \caption{The relation between intrinsic and extrinsic Dirac operators}\label{fig:exandin}
\end{figure}
For each face $\Delta_i$ the hyperplane perpendicular to $n_i$ gives a affine structure $\mathrm{Affine}(\Delta_i)$, we then can identify the edge $e_{ij}$ by
\begin{equation}
e_{ij}^i=\lvert E_{ij} \rvert \frac{df_{ij}-\langle df_{ij},n_i\rangle n_i}{\lvert df_{ij}-\langle df_{ij},n_i\rangle n_i\rvert }.
\label{eqn:affine_struc}
\end{equation}
Fix a reference frame $p^i$ for $\mathrm{Affine}(\Delta_i)$ and then $e_{ij}^i$ can be represented with $p^i$, denoted by $ \mathfrak{e}_{ij}^i$. \\
Recall that in the smooth case there  is  a section of the spinor bundle $\mathcal{S}\rightarrow \R^3$ given by $\phi_c=(c,1)$ where $c$ is the globally parallel section of the spin bundle. An immersion of the surface $X\hookrightarrow \R^3$ induces a section of $\mathcal{S}\rightarrow X$ by restricting $\phi_c$ on $X$. \\
Now choose a unit quaternion $g_i\in \mathrm{Spin}(3)$ such that 
\[e_{ij}^i=g_i^{-1}\cdot \mathfrak{e}_{ij}^i \cdot g_i\] 
The constant section of the spin bundle  can be formally defined by
\[c =p_i \cdot g_i\]
Then we can rewrite the spinor field $(c,1)$ as 
\begin{align*}
    (c, 1) & = (p_i\cdot g_i, 1) \\
    &= (p_i, g_i).
\end{align*}
The spinor connection is then given by 
\[g_{ij}=g_{i}\cdot h_{ij}\cdot g_{j}^{-1}\]
where $h_{ij}$ is defined  in \cref{lem:h} with $E_{ij}=e_{ij}^i  \cdot h_{ij}$.
The Dirac equation yields:
\begin{align*}
    2D(\phi_c) &= \sum\limits_j (p_i, \mathfrak{e}^i_{ij})\cdot\tilde{\eta}_{ji} (c, 1) %
    =\sum\limits_j (p_i,\mathfrak{e}_{ij}^i)\cdot \tilde{\eta}_{ji}(p_j\cdot g_j, 1) \\ &
    =\sum\limits_j (p_i,\mathfrak{e}_{ij}^i)\cdot \tilde{\eta}_{ji}(p_j,g_j)  %
    =\sum\limits_j (p_i,\mathfrak{e}_{ij}^i)\cdot (p_i, g_{ij}\cdot g_j)\\
    &= \sum\limits_j (p_i, \mathfrak{e}_{ij}^i\cdot g_{ij}\cdot g_j)%
    =\sum\limits_j (p_i,g_i\cdot e_{ij}^i\cdot g_i^{-1}\cdot g_{ij}\cdot g_j) \\
    &=\sum\limits_j (p_i,g_i\cdot e_{ij}^i\cdot h_{ij}) %
     = \sum\limits_j (p_i, g_i\cdot E_{ij}) \\
    &=(p_i, g_i \cdot \big(\sum\limits_j E_{ij}\big)) %
    =2\mathbf{H}_i \cdot (p_i,g_i) =2 \mathbf{H}_i \cdot (c, 1)\\
                &= 2\mathbf{H}_i\cdot \phi_c.
\end{align*}
It shows that the section $\phi_c$ satisfies the Dirac equation and the induced face-edge-constraint realization exactly recovers the original face-edge-constraint net. \\
Let $\mathcal{H}$ be functions from the faces to $\H$ and $\Gamma(\mathcal{S})$ be the spaces of the sections of the spinor bundle. The  map $\mathfrak{c}$ is constructed by:
\begin{align*}
\mathfrak{c}:\Gamma(\mathcal{S}) &\rightarrow \mathcal{H}\\
(c,\phi_i) &\mapsto \phi_i.
\end{align*}
The arguments above also imply that
\[\mathfrak{c}\circ (D-\mathbf{H})\circ \mathfrak{c}^{-1} = D_f.\]
Compared with \cref{thm:relation} this shows that the discretization of the both operators preserves the relation of their smooth correspondence. Note that with the affine structure \eqref{eqn:affine_struc} the intrinsic Dirac operator is different from the one in \cite{ye_2018} by a cosine factor, which was introduced for the purpose of numerics, because in that case the Dirac operator would be covariant under edge-length preserving deformations. In our case, the intrinsic Dirac operator is covariant under hyperedge-length preserving deformations and hence it is more consistent with the extrinsic one (\cref{fig:exandin}).

In summary, the key properties of our discrete extrinsic and intrinsic Dirac equations are that they both determine the local closing condition of an immersed surface in $\mathbb{R}^3$ and its mean curvature half-density. Besides, the notion of minimal surfaces in our framework generalize several well-known versions discrete minimal surfaces, coming from integrable systems and area variation formulations, respectively. The equivalence relation, induced by our discrete spin transformation, preserves many important properties of the spin structure from the smooth case.

\printbibliography
\end{document}